\documentclass[11pt,a4paper,leqno]{amsart}

\usepackage[all]{xy}

\usepackage[utf8]{inputenc}
\usepackage{amssymb}
\usepackage{amsmath}
\usepackage{amsthm}
\usepackage{varioref}
\usepackage{dsfont}
\usepackage{cases}
\usepackage{pict2e}
\usepackage{stmaryrd}
\usepackage{graphicx}
\usepackage{enumerate}
\RequirePackage{yhmath}

\DeclareMathOperator{\Span}{\textrm{span}}

\usepackage[a4paper]{geometry}
\begin{document}

\title{$n$-supercyclic and strongly $n$-supercyclic operators in finite dimensions}
\author{ERNST Romuald}
\date{}

\begin{abstract}
We prove that on $\mathbb{R}^N$, there is no $n$-supercyclic operator with $1\leq n< \lfloor \frac{N+1}{2}\rfloor$ i.e. if $\mathbb{R}^N$ has an $n$-dimensional subspace whose orbit under $T\in\mathcal{L}(\mathbb{R}^N)$ is dense in $\mathbb{R}^N$, then $n$ is greater than $\lfloor\frac{N+1}{2}\rfloor$. Moreover, this value is optimal.
We then consider the case of strongly $n$-supercyclic operators. An operator $T\in\mathcal{L}(\mathbb{R}^N)$ is strongly $n$-supercyclic if $\mathbb{R}^N$ has an $n$-dimensional subspace whose orbit under $T$ is dense in $\mathbb{P}_n(\mathbb{R}^N)$, the $n$-th Grassmannian. We prove that strong $n$-supercyclicity does not occur non-trivially in finite dimension.
\end{abstract}

\maketitle

\theoremstyle{plain}
\newtheorem{theo}{Theorem }[section]
\newtheorem{lem}[theo]{Lemma}
\newtheorem{cor}[theo]{Corollary}
\newtheorem{prop}[theo]{Proposition}

\theoremstyle{definition}
\newtheorem{defi}[theo]{Definition}
\newtheorem{rem}[theo]{Remark}

\newtheorem{exe}[theo]{Example}
\newtheorem{cexe}[theo]{Counterexample}
\newtheorem*{nota}{Notation}
\newtheorem*{quest}{Question}

%

\newtheorem*{thmbfs}{Theorem Bourdon, Feldman, Shapiro}
\newtheorem*{thmf}{Theorem Feldman}
\newtheorem*{thmme}{Theorem}
\newtheorem*{thmme1}{Theorem 1}
\newtheorem*{thmme2}{Theorem 2}
\newtheorem*{preuvetheo}{\textit{Proof of Proposition \ref{propfortedualite}}}
\newtheorem*{propfrom}{Proposition (Proposition 1.13 \cite{Ernststrongnsupop})}

\newcommand{\R}{\mathbb R}
\newcommand{\Z}{\mathbb Z} 
\newcommand{\N}{\mathbb N}
\newcommand{\Q}{\mathbb Q}
\newcommand{\C}{\mathbb C}
\newcommand{\K}{\mathbb K}
\newcommand{\Pn}{\mathbb P}
\newcommand{\HC}{\mathcal{H}\mathcal{C}}
\newcommand{\SC}{\mathcal{S}\mathcal{C}}
\newcommand{\Oc}{\mathcal{O}}
\newcommand{\ES}{\mathcal{E}\mathcal{S}}
\newcommand{\epsi}{\varepsilon}

Let $T$ be a continuous linear operator on a Banach space $X$.
The orbit of a set $E\in X$ under $T$ is defined by $$\Oc(E,T):=\cup_{n\in\Z_{+}}T^{n}(E).$$
Many authors have already studied some density properties of such orbits for different original sets $E$.
If $E$ is a singleton and $\Oc(E,T)$ is dense in $X$, then $T$ is said to be hypercyclic. Hypercyclicity has been first studied by Birkhoff in 1929 and has been a subject of great interest during the last twenty years, see \cite{Bay} and \cite{Grope} for a survey on hypercyclicity. Later, in 1974, Hilden and Wallen \cite{Hilwal} worked on a different set $E=\K x$ which is a one-dimensional subspace of $X$, and if $\Oc(E,T)$ is dense in $X$, then $T$ is said to be supercyclic.
Several generalisations of supercyclicity were proposed since like the one introduced by Feldman \cite{Felnsupop} in 2002. Rather than considering orbits of lines, Feldman defines an $n$-supercyclic operator as being an operator for which there exists an $n$-dimensional subspace $E$ such that $\Oc(E,T)$ is dense in $X$.
This notion has been mainly studied in \cite{Bayhyp}, \cite{Bou} and \cite{Felnsupop}.
In 2004, Bourdon, Feldman and Shapiro proved in the complex case that non-trivial $n$-supercyclicity is purely infinite dimensional:
\begin{thmbfs}
 Let $N\geq2$. Then there is no $(N-1)$-supercyclic operator on $\C^N$. In particular, there is no $n$-supercyclic operator on $\C^N$ for $1\leq n\leq N-1$.
\end{thmbfs}
The last theorem extends a result proved by Hilden and Wallen for supercyclic operators in the complex setting. On the other hand, Herzog proved that there is no supercyclic operators on $\R^n$ for $n\geq3$ in \cite{Herzog}. Therefore, it is natural to ask the question of the existence of $n$-supercyclic operators in the real setting.

In 2008, Shkarin introduced another generalisation of supercyclicity in \cite{Shkauniv}. Roughly speaking, an operator $T\in\mathcal{L}(X)$ is strongly $n$-supercyclic if there exists a subspace of dimension $n$ whose orbit is dense in the set of $n$-dimensional subspaces of $X$. To be more precise, we need to define the topology of this set, which is called the $n$-th-Grassmannian of $X$.
If $\dim(X)\geq n$ then one may define a topology on the $n$-th Grassmannian. To do this, let us consider the open subset $X_{n}$ of all linearly independent $n$-tuples with the topology induced from $X^{n}$ and let $\pi_{n}:X_{n}\to\Pn_{n}(X)$ be defined by $\pi_{n}(x)=\Span\{x_{1},\ldots,x_{n}\}$.
The topology on $\Pn_n(X)$ is the coarsest topology for which the map $\pi_n$ is open and continuous.
Let us now turn to the definition of strong $n$-supercyclicity: $M\in\Pn_{n}(X)$ is a strongly $n$-supercyclic subspace for $T$ if every $T^{k}(M)$ is $n$-dimensional and if $\{T^{k}(M),k\in\Z_{+}\}$ is dense in $\Pn_{n}(X)$.
If such a subspace exists, then $T$ is said to be strongly $n$-supercyclic. We denote by $\ES_n(T)$ the set of strongly $n$-supercyclic subspaces for an operator $T$.

An open question regarding $n$-supercyclic operators is to know whether they satisfy the Ansari property: is it true that $T^p$ is $n$-supercyclic for any $p\geq2$ provided $T$ itself is $n$-supercyclic?
Shkarin \cite{Shkauniv} has shown that strongly $n$-supercyclic operators do satisfy the Ansari property and he asks if $n$-supercyclicity and strong $n$-supercyclicity are equivalent. Indeed, this would solve the Ansari problem for $n$-supercyclic operators.
Unfortunately, Shkarin did not go further in the study of strongly $n$-supercyclic operators. A study of general properties of strongly $n$-supercyclic operators can be found in \cite{Ernststrongnsupop}.

In this paper, we study in details $n$-supercyclicity and strong $n$-supercyclicity in finite dimensional spaces. Of course, by the results of Bourdon, Feldman and Shapiro, we need only to concentrate on the real Banach spaces case.
In particular, in section 2 we are going to prove the following theorem:
\begin{thmme1}
 Let $N\geq2$. There is no $(\lfloor \frac{N+1}{2}\rfloor-1)$-supercyclic operators on $\R^N$. Moreover there exist $(\lfloor \frac{N+1}{2}\rfloor)$-supercyclic operators on $\R^N$.
\end{thmme1}
This theorem generalises Hilden and Wallen's and Herzog's results and is optimal for these operators. Actually, it is not difficult to prove that there exists an operator which is $k$-supercyclic but not $(k-1)$-supercyclic on $\R^N$ for every $\lfloor \frac{N+1}{2}\rfloor\leq k\leq N$.
The proof of Theorem 1 is not easy and one needs to get familiar with specific notations to fully understand it. The proof is progressing by steps from simplest matrices, which will be called primary, to general ones.

Then, in Section 3, we completely solve the question of the existence of non-trivial strongly $n$-supercyclic operators in finite dimensional vector spaces.
In fact, we prove:
\begin{thmme2}
 For $N\geq3$, there is no strongly $n$-supercyclic operator on $\R^{N}$ for $1\leq n< N$.
\end{thmme2}
This result puts an end to the study of strong $n$-supercyclicity in finite dimension. 
This proves, in particular, that there exist $n$-supercyclic operators that are not strongly $n$-supercyclic and answers the question of the equivalence between $n$-supercyclicity and strong $n$-supercyclicity raised in \cite{Shkauniv}.
The interested reader shall refer to \cite{Ernststrongnsupop} for other properties on strongly $n$-supercyclic operators in the infinite dimensional spaces setting.

\section{Preliminaries}

It has been known for years that in the real setting, supercyclic operators are completely characterised and they only appear on $\R$ or $\R^2$. Moreover, on $\R^2$, if $\pi$ and $\theta$ are linearly independent over $\Q$, then $R_\theta$, the rotation with angle $\theta$, is supercyclic.
Building on this, one may easily see that any rotation on $\R^3$ around any one-dimensional subspace and with angle linearly independent with $\pi$ over $\Q$ is 2-supercyclic. This simple example proves that the real setting is completely different from the complex one and gives hope in finding similar examples in higher dimensions.
It seems clear that rotations are making the difference between the real case and the complex case.
The next part is devoted to the Jordan real decomposition and highlights the role played by rotations in the real setting.

\begin{center}
 \textbf{Jordan decomposition}
\end{center}

In the complex setting, it is common to use the Jordan decomposition to obtain a matrix similar to $T$ but with a better ``shape''. Bourdon, Feldman and Shapiro took advantage of this decomposition to prove that there is no $(N-1)$-supercyclic operator on $\C^N$.
Recall that a Jordan block with eigenvalue $\mu$ and of size $k$ is usually a $k\times k$ matrix with $\mu$ along the main diagonal, ones on the first super-diagonal and zeros everywhere else.
For convenience, all along this paper we follow another convention which improves slightly the notations but does not change the efficiency of this decomposition. Thus, in our convention, a classical Jordan block with eigenvalue $\mu$ and of size $k$ will be a $k\times k$ matrix with $\mu$ along the main diagonal and along the first super-diagonal and zeros elsewhere.

This well-known decomposition for complex matrices cannot be applied without changes to the case of real matrices because of the existence of complex eigenvalues.
However, there also exists a real version of the Jordan decomposition which is an improvement of the last one. In the real case, every matrix is similar to a direct sum of classical Jordan blocks and real Jordan blocks, where a real Jordan block of modulus $\mu$ and of size $k$ is usually a $2k\times 2k$ matrix with $\mu R_\theta$ along the main diagonal, identity matrices along the first super-diagonal and zeros elsewhere. For the same reasons, our convention is different and for us the terms along the first super-diagonal are the same that those on the main diagonal i.e. $\mu R_\theta$.

Let $\mathcal{B}$ be a classical (respectively real) Jordan block with eigenvalue (respectively modulus) $\mu$ and of size $k$ and let $\mathcal{A}=\mu$ (respectively $\mathcal{A}=\mu R_\theta$). Then, powers of $\mathcal{B}$ are simple to compute.
Indeed, for all $n\in\N$,
$$\mathcal{B}^n=\left(\begin{array}{cccccc}
\mathcal{A}^n&\binom{n}{1}\mathcal{A}^n&\binom{n}{2}\mathcal{A}^n&\cdots&\binom{n}{k-1}\mathcal{A}^n\\
0&\mathcal{A}^n&\binom{n}{1}\mathcal{A}^n&\binom{n}{2}\mathcal{A}^n\cdots&\binom{n}{k-2}\mathcal{A}^n\\
0&0&\ddots&\ddots&\vdots\\
\vdots&&\ddots&\ddots&\binom{n}{1}\mathcal{A}^n\\
0&\cdots&0&0&\mathcal{A}^n\\
\end{array}
\right).$$
Due to the fact that we are going to use repeatedly the Jordan decomposition, in both real and complex cases, we use the term modulus instead of eigenvalue. If the reader wishes more informations on the Jordan decomposition see \cite{Hirsmal} or \cite{Mne} for a good review.

All along this paper, we will be interested in studying dynamical properties of such matrices. Consequently, we assume here for the whole paper that when we consider a Jordan block, its modulus is supposed to be non-zero.
To summarise, every operator on $\R^N$ is similar to one with the following shape:

$$\left(\begin{array}{ccccccccc}
\cline{1-1}
\multicolumn{1}{|c|}{J_{1}}&0&\cdots&\cdots&\cdots&0 \\ \cline{1-1}
0 &\ddots&0&\cdots&\cdots&0 \\ \cline{3-3}
0 &0&\multicolumn{1}{|c|}{J_{q}}&0&\cdots&0 \\ \cline{3-4}
0&\cdots&0&\multicolumn{1}{|c|}{\mathcal{J}_{1}}&0&0 \\ \cline{4-4}
0&\cdots&\cdots&0&\ddots&0 \\ \cline{6-6}
0&0&0&0&0&\multicolumn{1}{|c|}{\mathcal{J}_{r}} \\ \cline{6-6}
\end{array}\right)$$
where $J_{i}$ are classical Jordan blocks
$$J_{i}=\left(\begin{array}{ccccc}
\mu_{i}   &\mu_{i}&0  &0 \\
0&\ddots &\ddots&\cdots\\
\vdots&0&\ddots&\mu_{i} \\
0 & \cdots      &0&\mu_{i} \\
\end{array}\right)$$
and $\mathcal{J}_{i}$ are real Jordan blocks
$$\mathcal{J}_{i}=\left(\begin{array}{cccc}
\multicolumn{1}{c|}{\lambda_iR_{\theta_{i}}}   &\multicolumn{1}{|c|}{\lambda_iR_{\theta_{i}}}&0  &0 \\ \cline{1-2}
0&\ddots &\ddots&\cdots\\ \cline{4-4}
\vdots&0&\ddots&\multicolumn{1}{|c}{\lambda_iR_{\theta_{i}}} \\ \cline{4-4}
0 & \cdots      &0&\multicolumn{1}{|c}{\lambda_iR_{\theta_{i}}} \\
\end{array}\right).$$

\section{$n$-supercyclic operators on $\R^N$}

\subsection{Introduction}

Bourdon, Feldman and Shapiro showed that there are $n$-supercyclic operators on $\C^N$ if and only if $n=N$. This completely characterises $n$-supercyclic operators in the complex finite dimensional setting.
In this section, we are going to apply the real Jordan decomposition to determine for which $n\in\N$ there are $n$-supercyclic operators on $\R^N$.

Actually, the following examples reveal how to provide $(\lfloor \frac{N+1}{2}\rfloor)$-supercyclic operators on $\R^N$.
\begin{exe}\label{exensup}
For all $N\geq1$:

$\bullet$ On $\mathbb{R}^{2N}$, endomorphisms represented by matrices of the form
\begin{equation*}
\left(\begin{array}{cccc}
\multicolumn{1}{c|}{R_{\theta_{1}}}   &0 &\cdots  &0 \\ \cline{1-1}
  0                                   &\ddots &     0    & 0  \\ \cline{4-4}
\vdots & \cdots      & 0&\multicolumn{1}{|c}{R_{\theta_{N}}} \\
\end{array}\right)
\end{equation*}
 are $N$-supercyclic if (and only if) $\{\pi,\theta_{1},\ldots,\theta_{N}\}$ is a linearly independent family over $\mathbb{Q}$.

$\bullet$ On $\mathbb{R}^{2N+1}$, endomorphisms represented by matrices of the form
\begin{equation*}
\left(\begin{array}{ccccc}
\multicolumn{1}{c|}{R_{\theta_{1}}}   &0 &\cdots  &0&0 \\ \cline{1-1}
  0                                   &\ddots &     0    & 0&0  \\ \cline{4-5}
\vdots & \cdots      & 0&\multicolumn{1}{|c}{R_{\theta_{N}}}&0 \\
0&\cdots&0&\multicolumn{1}{|c}{0}&1\\
\end{array}\right)
\end{equation*}
are $(N+1)$-supercyclic if (and only if) $\{\pi,\theta_{1},\ldots,\theta_{N}\}$ is a linearly independent family over $\mathbb{Q}$.

The proof of this example relies on the fact that every rotation sub-matrix is supercyclic and the Kronecker density theorem \cite{Hardwright} permits to consider each one separately.
\end{exe}
These simple examples prove that our Theorem 1 is optimal. In the following, we are going to study $n$-supercyclic operators on $\R^N$ in order to prove Theorem 1. 
We progress step by step considering particular cases until we reach the remaining part of Theorem 1 in the general case.
We begin by proving two special cases: the case of a real Jordan block matrix of size 2 is considered first because it is the simplest matrix that Bourdon, Feldman and Shapiro have not checked in \cite{Bou} and then the case of a direct sum of rotation matrices because it permits to notice that something more is needed if one wants to go further.
These two results are stated and proved first because their proofs introduce some techniques involved for more general proofs.
Then, we will give a useful basis reduction which is of constant use all along the paper. From that point, our aim will be to find the best supercyclic constant for different types of matrices.
We will begin by primary matrices which are direct sums of unimodular real and complex Jordan blocs of size one and we will continue with the case of a single real Jordan block of arbitrary size.
After that, we discuss the best supercyclic constant for matrices being direct sums of Jordan blocks with pairwise different moduli and then for matrices being direct sums of Jordan blocks with the same modulus.
Finally, we gather these two last results in the last subsection to give a general result having Theorem 1 as a corollary.\\
Let us begin with a real Jordan block of size 2.

\begin{prop}\label{propjordpas2}
$ T=\left(\begin{array}{cc}
        R_{\theta}&R_{\theta}\\
	0&R_{\theta}\\
       \end{array}\right)$ is not 2-supercyclic on $\R^{4}$.
\end{prop}

\begin{proof}
Suppose that $T$ is 2-supercyclic to obtain a contradiction.
Let $M=\Span\{x,y\}$ be a 2-supercyclic subspace for $T$. Then, one can suppose either $x=(x_1,x_2,0,1)$ and $y=(y_1,y_2,1,0)$ or $x=(x_1,x_2,x_3,x_4)$ and $y=(y_1,y_2,0,0)$ where $(x_3,x_4)\neq(0,0)$.

$\bullet$ If $x=(x_1,x_2,0,1)$ and $y=(y_1,y_2,1,0)$, then for any non-empty open sets $U$ and $V$ in $\R^{2}$, there exist a strictly increasing sequence $(n_i)_{i\in\N}$ and two real sequences $(\lambda_{n_i})_{i\in\N},(\mu_{n_i})_{i\in\N}$ such that:

$$\begin{cases}\ 
R^{n_i}_{\theta}\left(\lambda_{n_i}\left(\begin{array}{c}x_1\\x_2\\\end{array}\right)+\mu_{n_i}\left(\begin{array}{c}y_1\\y_2\\\end{array}\right)\right) +n_iR^{n_i}_{\theta}\left(\begin{array}{c}\mu_{n_i}\\\lambda_{n_i}\\\end{array}\right)\in U\\
R^{n_i}_{\theta}\left(\begin{array}{c}\mu_{n_i}\\\lambda_{n_i}\\\end{array}\right)\in V
\end{cases}$$
which is equivalent to:
\begin{numcases}\ 
\lambda_{n_i}\left(\begin{array}{c}x_1\\x_2\\\end{array}\right)+\mu_{n_i}\left(\begin{array}{c}y_1\\y_2\\\end{array}\right) +n_i\left(\begin{array}{c}\mu_{n_i}\\\lambda_{n_i}\\\end{array}\right)\in R^{-n_i}_{\theta}\left(U\right)\label{num1+1}\\
\left(\begin{array}{c}\mu_{n_i}\\\lambda_{n_i}\\\end{array}\right)\in R^{-n_i}_{\theta}\left(V\right)\label{num1+2}
\end{numcases}\\
Let $V=B\left(\left(\begin{array}{c}1\\0\\\end{array}\right),\epsi\right)$ be an open ball of radius $\epsi$ centred in $\left(\begin{array}{c}1\\0\\\end{array}\right)$ with $0<\epsi<1$ and $U$ be any non-empty bounded open set, then (\ref{num1+2}) implies that for all $i\in\N$, $0\leq\vert\lambda_{n_i}\vert,\vert\mu_{n_i}\vert<1+\epsi$.
One may divide (\ref{num1+1}) by $n_i$ to get:
$$\frac{\lambda_{n_i}}{n_i}\left(\begin{array}{c}x_1\\x_2\\\end{array}\right)+\frac{\mu_{n_i}}{n_i}\left(\begin{array}{c}y_1\\y_2\\\end{array}\right) +\left(\begin{array}{c}\mu_{n_i}\\\lambda_{n_i}\\\end{array}\right)\in \frac{R^{-n_i}_{\theta}\left(U\right)}{n_i}$$
However, since the sequences $(\lambda_{n_i})_{i\in\N}$ and $(\mu_{n_i})_{i\in\N}$ are bounded, $\frac{\lambda_{n_i}}{n_i}\underset{i\to+\infty}{\longrightarrow}0$ and $\frac{\mu_{n_i}}{n_i}\underset{i\to+\infty}{\longrightarrow}0$ and since $U$ is a bounded set, $(\lambda_{n_{i}})_{i\in\N}$ and $(\mu_{n_{i}})_{i\in\N}$ have to go to zero. This contradicts $\left(\begin{array}{c}\mu_{n_i}\\\lambda_{n_i}\\\end{array}\right)\in R^{-n_i}_{\theta}\left(V\right)$.

$\bullet$ If $x=(x_1,x_2,x_3,x_4)$ and $y=(y_1,y_2,0,0)$, then one may suppose $\Vert(x_3,x_4)\Vert=1$. By 2-supercyclicity of $T$, for any non-empty open sets $U,V$ in $\R^{2}$, there exist a strictly increasing sequence $(n_i)_{i\in\N}$ and two real sequences $(\lambda_{n_i})_{i\in\N},(\mu_{n_i})_{i\in\N}$ such that:

$$\begin{cases}\ 
R^{n_i}_{\theta}\left(\lambda_{n_i}\left(\begin{array}{c}x_1\\x_2\\\end{array}\right)+\mu_{n_i}\left(\begin{array}{c}y_1\\y_2\\\end{array}\right)\right) +n_i\lambda_{n_{i}}R^{n_i}_{\theta}\left(\begin{array}{c}x_3\\ x_4\\\end{array}\right)\in U\\
\lambda_{n_i}R^{n_i}_{\theta}\left(\begin{array}{c}x_3\\x_4\\\end{array}\right)\in V
\end{cases}$$
which can be rewritten:
\begin{numcases}\ 
\lambda_{n_i}\left(\begin{array}{c}x_1\\x_2\\\end{array}\right)+\mu_{n_i}\left(\begin{array}{c}y_1\\y_2\\\end{array}\right) +n_i\lambda_{n_i}\left(\begin{array}{c}x_3\\x_4\\\end{array}\right)\in R^{-n_i}_{\theta}\left(U\right)\label{num2+1}\\
\lambda_{n_i}\left(\begin{array}{c}x_3\\x_4\\\end{array}\right)\in R^{-n_i}_{\theta}\left(V\right)\label{num2+2}
\end{numcases}
Let $V=B\left(\left(\begin{array}{c}r\\0\\\end{array}\right),\epsi\right)$ be an open ball of radius $\epsi$ centred in $\left(\begin{array}{c}r\\0\\\end{array}\right)$ with $0<\epsi<1$ and $r>1$.
According to (\ref{num2+2}), for every $i\in\N$ we have $r-\epsi<\vert \lambda_{n_i}\vert<r+\epsi$. Divide then (\ref{num2+1}) by $n_i \lambda_{n_{i}}$:
$$\frac{\mu_{n_i}}{n_i\lambda_{n_i}}\left(\begin{array}{c}y_1\\y_2\\\end{array}\right) +\left(\begin{array}{c}x_3\\x_4\\\end{array}\right)
                                                                                                                                                                                   \underset{i\to+\infty}{\longrightarrow}\left(\begin{array}{c}0\\0
                                                                                                                                                                                                                                \end{array}\right)
$$
From this we deduce that the sequence $\left(\frac{\mu_{n_i}}{n_i\lambda_{n_i}}\right)_{i\in\N}$ is convergent to some $t\in\R$ because $(y_1,y_2)\neq(0,0)$, so we have:
$$t\left(\begin{array}{c}y_1\\y_2\\\end{array}\right) +\left(\begin{array}{c}x_3\\x_4\\\end{array}\right)=\left(\begin{array}{c}0\\0\end{array}\right).$$
As $\left(\begin{array}{c}x_3\\x_4\\\end{array}\right)$ is non-zero, this last equation implies that $\left\{\left(\begin{array}{c}x_3\\x_4\\\end{array}\right),\left(\begin{array}{c}y_1\\y_2\\\end{array}\right)\right\}$ is linearly dependent.
Thus choosing an appropriate linear combination of $x$ and $y$, one may assume $x=(x_1,x_2,x_3,x_4)$ and $y=(x_3,x_4,0,0)$, hence (\ref{num2+1}) and (\ref{num2+2}) give:
\begin{numcases}\ 
\lambda_{n_i}\left(\begin{array}{c}x_1\\x_2\\\end{array}\right)+(\mu_{n_i} +n_i\lambda_{n_i})\left(\begin{array}{c}x_3\\x_4\\\end{array}\right)\in R^{-n_i}_{\theta}\left(U\right)\label{num3+1}\\
\lambda_{n_i}\left(\begin{array}{c}x_3\\x_4\\\end{array}\right)\in R^{-n_i}_{\theta}\left(V\right)\label{num3+2}
\end{numcases} 
Now, it is clear that the vectors $\left(\begin{array}{c}x_1\\x_2\\\end{array}\right)$ and $\left(\begin{array}{c}x_3\\x_4\\\end{array}\right)$ are linearly independent.
Indeed, suppose in order to obtain a contradiction that they are linearly dependent. Then upon taking appropriate linear combinations and replacing $x$, we can write $x=(0,0,x_3,x_4)$, $y=(x_3,x_4,0,0)$ and 
$$\begin{cases}\ 
(\mu_{n_i}+n_i\lambda_{n_i})\left(\begin{array}{c}x_3\\x_4\\\end{array}\right)\in R^{-n_i}_{\theta}\left(U\right)\\
\lambda_{n_i}\left(\begin{array}{c}x_3\\x_4\\\end{array}\right)\in R^{-n_i}_{\theta}\left(V\right)
\end{cases}$$
But if one chooses two non-empty open sets $U$ and $V$ such that there does not exist a straight line passing through the origin and intersecting both $U$ and $V$, then we have a contradiction.

Thus $\left(\begin{array}{c}x_1\\x_2\\\end{array}\right)$ and $\left(\begin{array}{c}x_3\\x_4\\\end{array}\right)$ are linearly independent and let $\alpha$ denote the angle between these two vectors, then $\vert\sin(\alpha)\vert>0$. Choose $0<a<\vert\sin(\alpha)\vert(r-\epsi)\left\Vert\left(\begin{array}{c}x_1\\x_2\\\end{array}\right)\right\Vert$ and  $U=B\left(\small\left(\begin{array}{c}\frac{a}{2}\\0\end{array}\right),\frac{a}{4}\right)$ and let $\mathcal{C}_{U}$ denote the annulus obtained by rotations of the ball $U$.
With a little computation, one may easily notice that the set $\left(\R\left(\begin{array}{c}x_3\\x_4\\\end{array}\right)+[r-\epsi,r+\epsi]\left(\begin{array}{c}x_1\\x_2\\\end{array}\right)\right)$ does not intersect $\mathcal{C}_{U}$ contradicting (\ref{num3+1}) (see the figure below).
\\\includegraphics[width=15cm]{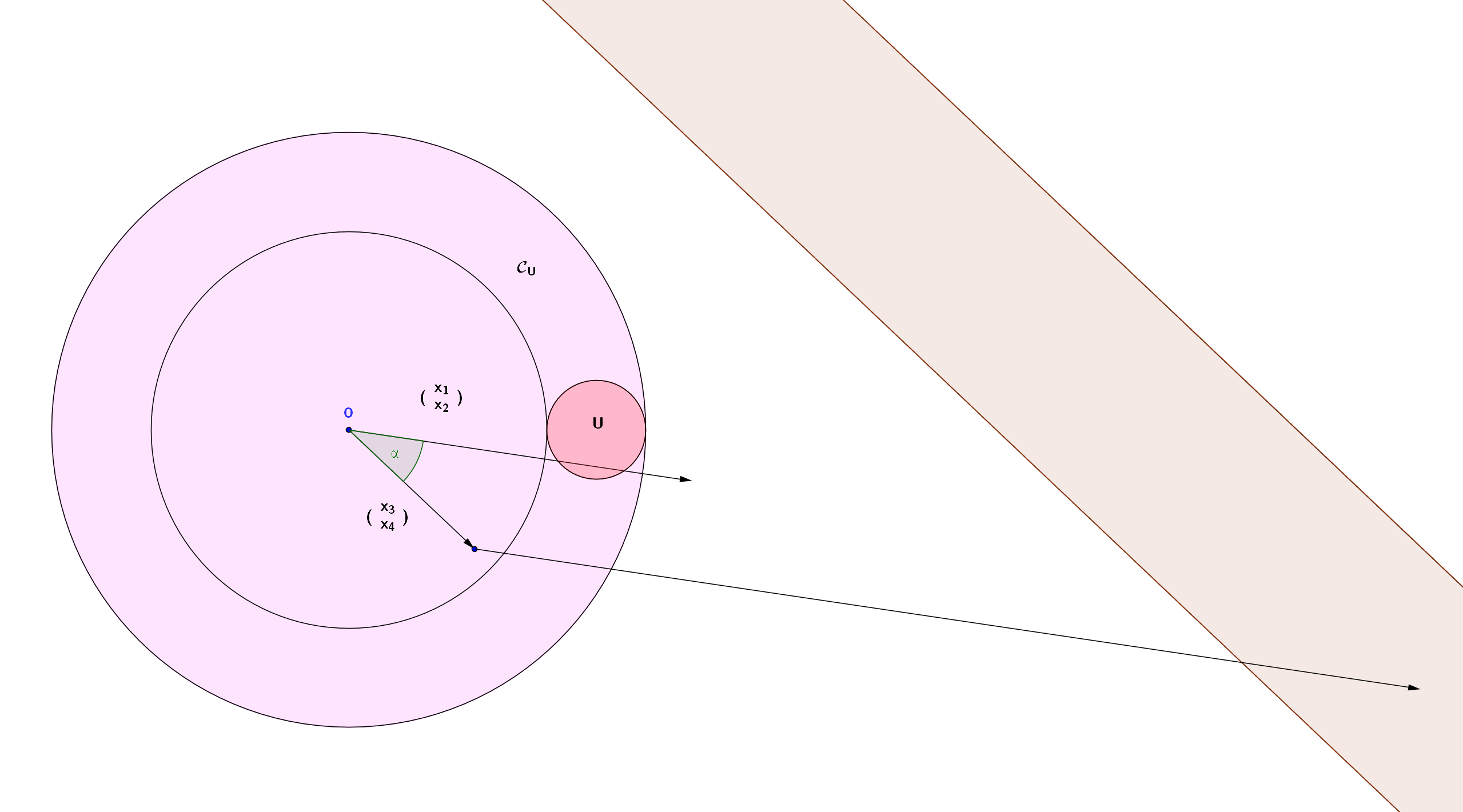}\\
So $ \left(\begin{array}{cc}
        R_{\theta}&R_{\theta}\\
	0&R_{\theta}\\
       \end{array}\right)$ is not 2-supercyclic.
\end{proof}

\begin{rem}
 One can easily notice that the previous matrix is 3-supercyclic if $\pi$ and $\theta$ are linearly independent over $\Q$.
\end{rem}

\begin{rem}
As one can notice, the previous proof is divided into two parts depending on the ''shape`` of the basis. Actually, to be able to deal with such operators, we will constantly make differences according to the basis' shape.
\end{rem}

\subsection{A leading example}

We deal with an example to show that we need some more tools if we want to go further in a precise manner. First, the next result proves that the supercyclic constants cannot be improved for the two matrices given in Example \ref{exensup} i.e. the first operator is not $(N-1)$-supercyclic and the second one is not $N$-supercyclic..
Moreover, in the following, $T$ is a direct sum of rotations' multiples, every one of these acting on $\R^2$. Hence when one usually consider a vector component, we consider a vector bi-component instead, meaning that for the next result the natural way to define a vector is not as being in $\R^{2N}$ but rather in $(\R^2)^N$. In the following, the $k$-th bi-component of a vector $(x_{1},\dots,x_{2N})$ is the vector on which the $k$-th rotation matrix acts i.e. the vector $(x_{2k-1},x_{2k})$.

\begin{prop}\label{propnrotpasn-1}
Let $N\geq2$, then $R_N:=\left(\begin{array}{ccccc}
a_{1}R_{\theta_1}&0&\cdots&0\\
0&\ddots&&\vdots\\
\vdots&&\ddots&0\\
0&\cdots&0&a_{N}R_{\theta_N}\\
\end{array}
\right)$ is not $(N-1)$-supercyclic on $\R^{2N}$ for every choice of $a_1,\ldots,a_N\in\R$ and every choice of $\theta_1,\ldots,\theta_N\in\R$.

\end{prop}

\begin{proof}
First upon reordering blocks in $R_N$ and taking a scalar multiple, one may suppose $0<\vert a_1\vert \leq\ldots\leq \vert a_{N-1}\vert\leq a_N=1$. Indeed, the nullity of one of the $a_i$ implies that $R_N$ has not dense range and is not $(N-1)$-supercyclic.
\medskip

We are going to prove that $R_N$ is not $(N-1)$-supercyclic on $\R^{2N}$ by induction.\\
For $N=2$, the result follows from Herzog's result \cite{Herzog}.\\
Suppose that for every $2\leq k< N$, every $\theta_1,\ldots,\theta_k$ and every $0<\vert a_1\vert\leq\ldots\leq\vert a_{k-1}\vert\leq a_{k}=1$ no matrix of the form $R_{k}$ is $(k-1)$-supercyclic. Let us prove it also for $R_{N}$.
Assume to the contrary that $R_N$ is $(N-1)$-supercyclic and let $M=\Span\{x^{1},\cdots,x^{N-1}\}$ be a $(N-1)$-supercyclic subspace for $R_N$. Define $x_{N+1}^{i}:=x^{i}$ for every $1\leq i\leq N-1$.

We argue that for every $k\in\{1, \ldots, N\}$, $M$ is spanned by a family of vectors $\{x_{k}^{1},\ldots,x_{k}^{N-1}\}$ such that if we define $p_{N+1}:=0$ and $$p_{k}:=\sup\left(j\in\{1,\ldots,N-1\}:\text{ the }k\text{-th bi-component from }x_{k}^{j}\text{ is non-null }\right)$$ then we have the following extra properties.
\begin{enumerate}[(a)]
 \item if $p_{k+1}\neq N-1$ then $p_k\in\{p_{k+1}+1,p_{k+1}+2\}$ and if $k\neq N$, $x_{k}^{j}\neq x_{k+1}^{j}$ for every $1\leq j\leq p_{k+1}$,\\
\item if $p_k=p_{k+1}+2$, then $x_{k}^{p_{k}-1}=\left(\begin{array}{c}
                                                         0\\1\\
                                                        \end{array}\right)$ and $x_{k}^{p_{k}}=\left(\begin{array}{c}
                                                         1\\0\\
                                                        \end{array}\right)$,\\
\item for every $k\leq l\leq N$ and every $p_l<j\leq N-1$, the $l$-th bi-component of the vector $x^{j}_{k}$ is null.\\
\end{enumerate}
We are going to prove this by decreasing induction on $k\in\{1,\ldots,N\}$.
Let's begin with the case $k=N$.
\\Upon taking appropriate linear combinations of basis elements of $M$ and reordering, one may assume that we have a basis $x_{N}^{1},\ldots,x_{N}^{N-1}$ of $M$ such that the last bi-component is non-zero either for $x_{N}^{1}$ (i.e. $p_{N}=1$) or for $x_{N}^{1}$ and $x_{N}^{2}$ (i.e. $p_{N}=2$) and is null for the other basis vectors. Moreover, in this last case, one may also require them to be $\left(\begin{array}{c}0\\1\\\end{array}\right)$ and $\left(\begin{array}{c}1\\0\\\end{array}\right)$ as in the proof of Proposition \ref{propjordpas2}.
One may easily notice that the induction hypothesis is satisfied for $k=N$.

Assume that the induction hypothesis is true for $N,\ldots,k+1$, let us check it for $k$.
\\Define $x_{k}^{j}=x_{k+1}^{j}$ for every $1\leq j\leq p_{k+1}$.
Upon taking appropriate linear combinations of the vectors $x_{k+1}^{p_{k+1}+1},\ldots,x_{k+1}^{N-1}$ and reordering one may get $N-1-p_{k+1}$ vectors $x_{k}^{p_{k+1}+1},\ldots,x_{k}^{N-1}$ with $\Span\{x_{k}^{1},\ldots,x_{k}^{N-1}\}=M$ satisfying one of the three following conditions:

\hspace{0,5cm}$\rhd$ the $k$-th bi-component of the vectors $x_{k}^{p_{k+1}+1},\ldots,x_{k}^{N-1}$ is null, i.e. $p_k=p_{k+1}$ but this yields a contradiction. Indeed, as $R_{N}$ is $(N-1)$-supercyclic there exists a strictly increasing sequence $(n_i)_{i\in\N}$ and $N-1$ real sequences $(\lambda_{1}^{(n_i)})_{i\in\N},\ldots,(\lambda_{N-1}^{(n_i)})_{i\in\N}$ such that:
\begin{equation}\label{eqiterR}
\left(\begin{array}{c}
a_{1}^{n_i}R_{\theta_1}^{n_i}\left(\sum_{j=1}^{N-1}\lambda_{j}^{(n_i)}\left(\begin{array}{c}x_{k}^{j}(1)\\x_{k}^{j}(2)\\\end{array}\right)
\right)\\
\vdots\\
a_{k}^{n_i}R_{\theta_{k}}^{n_i}\left(\sum_{j=1}^{p_{k}}\lambda_{j}^{(n_i)}\left(\begin{array}{c}x_{k}^{j}(2k-1)\\x_{k}^{j}(2k)\\\end{array}\right)\right)\\
\vdots\\
a_{N}^{n_i}R_{\theta_N}^{n_i}\left(\sum_{j=1}^{p_{N}}\lambda_{j}^{(n_i)}\left(\begin{array}{c}x_{k}^{j}(2N-1)\\x_{k}^{j}(2N)\\\end{array}\right)\right)\\
\end{array}\right)\underset{i\to\infty}{\longrightarrow}\left(\begin{array}{c}\left(\begin{array}{c}0\\0\end{array}\right)\\\vdots\\\left(\begin{array}{c}1\\0\end{array}\right)\\\vdots\\\left(\begin{array}{c}0\\0\end{array}\right)\\\end{array}\right)
\end{equation}
Then, the last bi-component above implies that for every $1\leq j\leq p_{N}$, $a_{N}^{n_i}\lambda_{j}^{(n_i)}\underset{i\to+\infty}{\longrightarrow}0$ because $R_{N}$ is an isometry, $1\leq p_{N}\leq 2$ and if $p_{N}=2$ then $\sum_{j=1}^{p_{N}}\lambda_{j}^{(n_i)}\left(\begin{array}{c}x_{k}^{j}(2N-1)\\x_{k}^{j}(2N)\\\end{array}\right)=\left(\begin{array}{c}\lambda_{2}^{(n_i)}\\\lambda_{1}^{(n_i)}\\\end{array}\right)$ by induction hypothesis.
Step by step, following the same idea, we can prove in the same way that for every $1\leq j\leq p_{k+1}$, $a_{k+1}\lambda_{j}^{(n_i)}\underset{i\to+\infty}{\longrightarrow}0$ because $\vert a_{k+1}\vert\leq\ldots\leq\vert a_{N}\vert$. Moreover if $p_j=N-1$ for some $j\in\{k+1,\ldots,N\}$ then we conclude at this step that for every $1\leq j\leq N-1$, $a_{k}^{n_i}\lambda_{j}^{(n_i)}\underset{i\to+\infty}{\longrightarrow}0$ contradicting (\ref{eqiterR}).
Since $p_k=p_{k+1}$ and $\vert a_k\vert\leq\vert a_{k+1}\vert$, for every $1\leq j\leq p_{k}$, $a_{k}^{n_i}\lambda_{j}^{(n_i)}\underset{i\to+\infty}{\longrightarrow}0$ but this contradicts (\ref{eqiterR}).

\hspace{0,5cm}$\rhd$ the $k$-th bi-component of the vectors $x_{k}^{p_{k+1}+2},\ldots,x_{k}^{N-1}$ is null but not for $x_{k}^{p_{k+1}+1}$. Then, $p_{k}=p_{k+1}+1$ and for every $p_{k}<j\leq N-1$, the $k$-th bi-component of the vector $x^{j}_{k}$ is null by construction and for every $k+1\leq l\leq N$ and every $p_l<j\leq N-1$, the $l$-th bi-component of the vector $x_{k}^{j}$ is also null
because $p_{k}>p_{k+1}>\ldots>p_N$ and the family $\{x_{k}^{p_{k+1}+1},\ldots,x_{k}^{N-1}\}$ is obtained by taking linear combinations of elements from the vectors $x_{k+1}^{p_{k+1}+1},\ldots,x_{k+1}^{N-1}$ whose $l$-th bi-component is null by induction hypothesis.

\hspace{0,5cm}$\rhd$ the $k$-th bi-component of the vectors $x_{k}^{p_{k+1}+3},\ldots,x_{k}^{N-1}$ is null but not for $x_{k}^{p_{k+1}+1}$ and $x_{k}^{p_{k+1}+2}$ and these two components can be chosen to be $\left(\begin{array}{c}0\\1\\\end{array}\right)$ and $\left(\begin{array}{c}1\\0\\\end{array}\right)$. Here, $p_{k}=p_{k+1}+2$ and we conclude as above.
\\This ends the induction process.

Let us denote by $y^{1},\ldots,y^{N-1}$ the vectors $x_{1}^{1},\ldots,x_{1}^{N-1}$ obtained thanks to the induction process. We proved that the sequence $(p_{N+1-k})_{0\leq k\leq N}$ is increasing until it reaches $N-1$ and is constant after and $p_{N+1}=0$ hence $p_2=N-1$.
This remark now permits to conclude.
Indeed, as $M$ is $(N-1)$-supercyclic subspace for $R_{N}$ then there exist a strictly increasing sequence $(n_i)_{i\in\N}$ and $N-1$ real sequences $(\lambda_{1}^{(n_i)})_{i\in\N},\ldots,(\lambda_{N-1}^{(n_i)})_{i\in\N}$ such that:
\begin{equation}\label{eqiterR2}\left(\begin{array}{c}
a_{1}^{n_i}R_{\theta_1}^{n_i}\left(\sum_{j=1}^{N-1}\lambda_{j}^{(n_i)}\left(\begin{array}{c}x_{1}^{j}\\x_{2}^{j}\\\end{array}\right)
\right)\\
a_{2}^{n_i}R_{\theta_2}^{n_i}\left(\sum_{j=1}^{N-1}\lambda_{j}^{(n_i)}\left(\begin{array}{c}x_{3}^{j}\\x_{4}^{j}\\\end{array}\right)
\right)\\
\vdots\\
a_{N}^{n_i}R_{\theta_{N}}^{n_i}\left(\sum_{j=1}^{p_N}\lambda_{j}^{(n_i)}\left(\begin{array}{c}x_{2N-1}^{j}\\x_{2N}^{j}\\\end{array}\right)\right)\\
\end{array}\right)\underset{i\to\infty}{\longrightarrow}\left(\begin{array}{c}\left(\begin{array}{c}1\\0\end{array}\right)\\\left(\begin{array}{c}0\\0\end{array}\right)\\\vdots\\\left(\begin{array}{c}0\\0\end{array}\right)\\\end{array}\right)
\end{equation}
On the basis of similar reasoning as in the induction process with $p_k=p_{k+1}$ we observe that for every $1\leq j\leq p_{2}=N-1$, $a_{2}^{n_i}\lambda_{j}^{(n_i)}\underset{i\to+\infty}{\longrightarrow}0$ and since $\vert a_1\vert\leq\vert a_{2}\vert$, for every $1\leq j\leq N-1$, $a_{1}^{n_i}\lambda_{j}^{(n_i)}\underset{i\to+\infty}{\longrightarrow}0$ but this contradicts (\ref{eqiterR2}).
\end{proof}
The key in the proof is the adaptation of the basis to the shape of $R_N$ and we are going to make constant use of this method in what follows. This motivates us to detail this method in the next part. 

\subsection{Basis reduction}

Let $m,N\in\N$, $T$ be a linear operator on $\R^N$, $\{x^1,\ldots,x^m\}$ be a linearly independent family in $\R^N$ and $M$ be the subspace spanned by this family.
Using the Jordan real decomposition one may suppose:

$$T=\left(\begin{array}{cccc}
a_{1}\mathcal{B}_{1}&0 &\cdots  &0\\
0&a_{2}\mathcal{B}_{2}&\ddots&\vdots\\
\vdots&&\ddots&0 \\
0&\cdots&0&a_{\gamma}\mathcal{B}_{\gamma} \\
\end{array}\right)$$

where $\mathcal{B}_i=\left(\begin{array}{ccccc}
\mathcal{A}_{i}&\mathcal{A}_i&0 &\cdots  &0\\
0&\mathcal{A}_{i}&\mathcal{A}_i&\ddots&\vdots\\
\vdots&&\ddots&\ddots&\\
&&&\ddots&\mathcal{A}_i \\
0&\cdots&&0&\mathcal{A}_{i} \\
\end{array}\right)$ is a classical or real Jordan block for any $1\leq i\leq \gamma$ with $\mathcal{A}_i=1$ or $\mathcal{A}_i=R_{\theta_i}$ respectively and $\gamma$ is the number of Jordan blocks in the decomposition of $T$.
Define $\tau_i=1$ when $\mathcal{B}_i$ is classical and $\tau_i=2$ when $\mathcal{B}_i$ is real and take also $\rho_i$ such that $\tau_i \rho_i$ is $\mathcal{B}_i$'s size, we will call $\rho_{i}$ the relative size of the block $\mathcal{B}_{i}$.
If $\mathcal{B}_{i}$ is a classical Jordan block, then its relative size is just its size, on the contrary if $\mathcal{B}_{i}$ is a real Jordan block, then its relative size is just its size divided by 2.
We will also denote by $\rho:=\sum_{i=1}^{\gamma}\rho_i$ the relative size of the matrix of $T$. Observe also that with these notations $N=\sum_{i=1}^{\gamma}\rho_i\tau_i$.

\begin{nota}
For the sake of clarity, we introduce a new notation before stating the following theorem.
Let $T$ be a linear operator on $\R^N$ in the previous Jordan form and $x\in\R^{N}$, we define for $1\leq i\leq \rho$
$$\chi_{i}(x)=\begin{cases}
x_{\sum_{l=1}^{p-1}\tau_l \rho_l+i-\sum_{l=1}^{p-1}\rho_l}&\text{ if }\tau_p=1\\
\left(\begin{array}{c}
x_{\sum_{l=1}^{p-1}\tau_l \rho_l+2(i-\sum_{l=1}^{p-1}\rho_l)-1}\\
x_{\sum_{l=1}^{p-1}\tau_l \rho_l+2(i-\sum_{l=1}^{p-1}\rho_l)}\\
\end{array}\right)&\text{ if }\tau_p=2\\
\end{cases}$$
where $p$ is the unique natural number satisfying: $\sum_{l=1}^{p-1}\rho_l<i\leq \sum_{l=1}^{p}\rho_l$. Roughly speaking $p$ is the number of the block $\mathcal{B}_p$ of $T$ which is acting on $\chi_{i}(x)$.
This probably seems a bit complicated at first sight but the underlying idea is natural: the operator $T$ is seen as almost a ''sum'' of operators $\mathcal{A}_i$ acting on either $\R$ or $\R^2$. Then it is also natural to consider the vectors $T$ is acting on, as a direct sum of vectors that the operators $\mathcal{A}_i$ are acting on.
To summarise, on some parts (classical) $T$ acts like if it was an operator on $\R$ and on the others (real), it acts as on $\R^2$, thus $\chi_{i}(x)$ may be either a scalar or a vector of size 2.
\\Let us explain this on an example.
Consider $T=\left(\begin{array}{ccccc}
a\mathcal{B}_{1}&0 &0\\
0&b\mathcal{B}_{2}&0\\
0&0&c\mathcal{B}_{3}\\
\end{array}\right)
=\left(\begin{array}{ccccc}
a&0 &0&0&0\\
0&bR_{\theta}&bR_{\theta}&0&0\\
0&0&bR_{\theta}&0&0\\
0&0&0&c&0\\
0&0&0&0&c
\end{array}\right)$ acting on $\R^7$ then $\tau_1=1,\tau_2=2,\tau_3=1,\rho_1=1,\rho_2=2,\rho_3=2$ and we shall decompose $x=(x_{1},x_{2},x_{3},x_{4},x_{5},x_{6},x_{7})$ as $x=\left(\begin{array}{c} \chi_1(x)\\\chi_2(x)\\\chi_3(x)\\\chi_4(x)\\\chi_5(x)\end{array}\right)$ with $\chi_{1}(x)=x_1,\chi_{2}(x)=\left(\begin{array}{c}x_2\\x_3\end{array}\right),\chi_{3}(x)=\left(\begin{array}{c}x_4\\x_5\end{array}\right),\chi_4(x)=x_6,\chi_5(x)=x_7$.
\end{nota}

Let us state the awaited theorem which is the main tool to prove the results announced at the beginning of the article.

\begin{theo}\label{theoreductiondebase}
Let $T$ be a linear operator on $\R^N$ in the Jordan form. Let also $M$ be an $m$-dimensional subspace.\\
Then, there exist a basis $\{y^1,\ldots,y^m\}$ of $M$, a non-decreasing sequence of integers $(\kappa_i)_{i\in\Z_{+}}$ and a sequence of sets $(\Lambda_i)_{i\in\Z_{+}}\subset \R\cup\R^2$ satisfying:
\begin{enumerate}[(a)]
 \item $\kappa_0=1$, $\Lambda_0=\{\chi_{\rho}(y^{j}), \kappa_0\leq j\leq m\}$.\label{eqredb}\\
 \item For every $i\in\Z_{+}$, $\kappa_{i+1}=\kappa_{i}+\dim(\Span\{\Lambda_i\})$ and $\Lambda_{i+1}=\{\chi_{\rho-(i+1)}(y^{j}), \kappa_{i+1}\leq j\leq m\}$.\label{eqredc}\\
 \item For every $i\in\{0,\ldots,\rho-1\}$, $\{\chi_{\rho-i}(y^{j}),\kappa_i\leq j< \kappa_{i+1}\}$ is either empty or linearly independent.\label{eqredd}\\
 \item $\kappa_\rho=m+1$.\label{eqrede}\\
 \item For every $p\in\{1,\ldots,\rho\}$ and every $j\in\{\kappa_p,\ldots,m\}$, $\chi_{\rho-p+1}(y^{j})=0$ or $\left(\begin{array}{c}
                                                        0\\0\\
                                                       \end{array}\right)$.\label{eqredf}
\end{enumerate}
\end{theo}

\begin{proof}
We want to construct a basis of $M$ adapted to the decomposition of $T$. Of course, this reduction heavily depends on $T$.
Let $x^1,\ldots,x^m$ be a basis of $M$.
We are going to create an increasing sequence of natural numbers $(\kappa_p)_{p\in\Z_{+}}$ and a sequence of sets $(\Lambda'_p)_{p\in\Z_{+}}$. For every step of the reduction, the sequence of natural numbers marks the vector number up to which the reduction has been completed and the sequence of sets contains the part of the vectors we have to reduce to the next step.
\\First define $\kappa_0=1$ and $\Lambda'_0=\{\chi_{\rho}(x^{i}), \kappa_0\leq i\leq m\}$.
By definition, $\Lambda'_{0}$ is either a subset of $\R$ or of $\R^2$, then $\dim(\Span\{\Lambda'_0\})=0,1$ or $2$.

$\bullet$ If $\dim(\Span\{\Lambda'_0\})=0$ then $\Vert \chi_{\rho}(x^{i})\Vert=0$ for any $1\leq i\leq m$ and we set $\kappa_1:=\kappa_0$ and $x_{1}^{j}:=x^{j}$ for every $\kappa_0\leq j\leq m$.

$\bullet$ If $\dim(\Span\{\Lambda'_0\})=1$, upon taking proper linear combinations of $x^1,\ldots,x^m$ and reordering, one may obtain a new basis $x_{1}^{1},\ldots,x_{1}^{m}$ of $M$ with $\Vert \chi_{\rho}(x_{1}^{1})\Vert=1$ and $\Vert \chi_{\rho}(x_{1}^{i})\Vert=0$ for any $\kappa_0+1\leq i\leq m$ and set $\kappa_1:=\kappa_0+1$.

$\bullet$ If $\dim(\Span\{\Lambda'_0\})=2$, upon taking proper linear combinations of $x^1,\ldots,x^m$, reordering, one may obtain a new basis $x_{1}^{1},\ldots,x_{1}^{m}$ of $M$ with $\chi_{\rho}(x_{1}^{1})=\left(\begin{array}{c}0\\1\\\end{array}\right)$, $\chi_{\rho}(x_{1}^{2})=\left(\begin{array}{c}1\\0\\\end{array}\right)$ and $\Vert \chi_{\rho}(x_{1}^{i})\Vert=0$ for $\kappa_0+2\leq i\leq m$ and set $\kappa_1:=\kappa_0+2$.
\\Then set also $\Lambda'_1=\{\chi_{\rho-1}(x_{1}^{i}), \kappa_1\leq i\leq m\}$, thus $\dim(\Span\{\Lambda'_1\})=0,1$ or $2$. Define $x_{2}^{i}=x_{1}^{i}$ for every $1\leq i<\kappa_1$.
Upon taking appropriate linear combinations of the vectors $x_{1}^{\kappa_1},\ldots,x_{1}^{m}$ and reordering one may get $m-\kappa_1+1$ vectors $x_{2}^{\kappa_1},\ldots,x_{2}^{m}$ with $\Span\{x_{2}^{1},\ldots,x_{2}^{m}\}=M$ satisfying one of the three following conditions:

$\bullet$ If $\dim(\Span\{\Lambda'_1\})=0$, then $\Vert \chi_{\rho-1}(x_{2}^{i})\Vert=0$ for any $\kappa_1\leq i\leq m$ and we set $\kappa_2:=\kappa_1$.

$\bullet$ If $\dim(\Span\{\Lambda'_1\})=1$, then $\Vert \chi_{\rho-1}(x_{2}^{\kappa_1})\Vert=1$ and $\Vert \chi_{\rho-1}(x_{2}^{i})\Vert=0$ for any $\kappa_1+1\leq i\leq m$ and set $\kappa_2:=\kappa_1+1$.

$\bullet$ If $\dim(\Span\{\Lambda'_1\})=2$, then $\chi_{\rho-1}(x_{2}^{\kappa_1})=\left(\begin{array}{c}0\\1\\\end{array}\right)$, $\chi_{\rho-1}(x_{2}^{\kappa_1+1})=\left(\begin{array}{c}1\\0\\\end{array}\right)$ and $\Vert \chi_{\rho-1}(x_{2}^{i})\Vert=0$ for $\kappa_1+2\leq i\leq m$ and set $\kappa_2:=\kappa_1+2$.

Suppose that this construction has been carried out until we obtain $x_{k}^{1},\ldots,x_{k}^m$, then set $\Lambda'_k=\{\chi_{\rho-k}(x_{k}^{i}), \kappa_k\leq i\leq m\}$, thus $\dim(\Span\{\Lambda'_k\})=0,1$ or $2$. Define $x_{k+1}^{i}=x_{k}^{i}$ for every $1\leq i<\kappa_k$.
Upon taking appropriate linear combinations of the vectors $x^{\kappa_k},\ldots,x^{m}$ and reordering one may get $m-\kappa_k+1$ vectors $x_{k+1}^{\kappa_k},\ldots,x_{k+1}^{m}$ with $\Span\{x_{k+1}^{1},\ldots,x_{k+1}^{m}\}=M$ satisfying one of the three following conditions:

$\bullet$ If $\dim(\Span\{\Lambda'_k\})=0$, then $\Vert \chi_{\rho-k}(x_{k+1}^{i})\Vert=0$ for any $\kappa_k\leq i\leq m$ and we set $\kappa_{k+1}:=\kappa_k$.

$\bullet$ If $\dim(\Span\{\Lambda'_k\})=1$, then $\Vert \chi_{\rho-k}(x_{k+1}^{\kappa_k})\Vert=1$ and $\Vert \chi_{\rho-k}(x_{k+1}^{i})\Vert=0$ for any $\kappa_k+1\leq i\leq m$ and set $\kappa_{k+1}:=\kappa_k+1$.

$\bullet$ If $\dim(\Span\{\Lambda'_k\})=2$, then $\chi_{\rho-k}(x_{k+1}^{\kappa_k})=\left(\begin{array}{c}0\\1\\\end{array}\right)$, $\chi_{\rho-k}(x_{k+1}^{\kappa_1+1})=\left(\begin{array}{c}1\\0\\\end{array}\right)$ and $\Vert \chi_{\rho-k}(x_{k+1}^{i})\Vert=0$ for $\kappa_k+2\leq i\leq m$ and set $\kappa_{k+1}:=\kappa_k+2$.

As a consequence, step by step we finally get a basis $(y^1,\ldots,y^m):=(x_{\rho}^{1},\ldots,x_{\rho}^{m})$ of $M$ and we set $\kappa_q:=\kappa_\rho$ for every $q$ greater than $\rho$.
We set $\Lambda_0=\{\chi_{\rho}(y^{j}), \kappa_0\leq j\leq m\}$ and $\Lambda_{i}=\{\chi_{\rho-i}(y^{j}), \kappa_{i}\leq j\leq m\}$, thus (\ref{eqredb}) is satisfied by definition. 
It suffices to remark then that $\dim(\Span\{\Lambda_{i}\})=\dim(\Span\{\Lambda'_{i}\})$ to check (\ref{eqredc}), (\ref{eqredd}) and (\ref{eqredf}).
Moreover,  (\ref{eqrede}) is also satisfied as $(y^{1},\ldots,y^m)$ form a basis of $M$ then $y^m$ is non-zero. Hence $\Lambda_\rho=\emptyset$ ($\Leftrightarrow \kappa_\rho=m+1$).
\end{proof}
\begin{rem}
The reduced basis we have described in the previous theorem has the following inverse staircase shape:
$$\tiny\left(\begin{array}{cccccccc}
&&&&&&&\multicolumn{1}{c}{\cdots} \\\cline{8-8}
&&&&&&\multicolumn{1}{c|}{\cdots}&\\\cline{7-7}
&{\fontsize{2cm}{1cm}\selectfont \text{*}}&&&&\multicolumn{1}{c|}{\adots}&& \\
&&&&&&& \\ 
&&\adots&&&{\fontsize{2cm}{1cm}\selectfont \text{0}}&& \\
&\adots&&&&&&\\ 
\multicolumn{1}{c|}{\cdots}&&&&&&&\\ \cline{1-1}
\end{array}\right)$$
We keep these notations for the rest of this paper. The reader needs to have in mind these notations when we decompose an operator in its Jordan form or when we reduce a basis. When we will need to refer to Theorem \ref{theoreductiondebase}, we will say that some basis has been reduced with respect to an operator.
\end{rem}

\begin{nota}
From now on, we will need to work with several vectors $x^{1},\ldots,x^{m}$. For this reason, we leave the heavy notation $\chi_{i}(x^{j})$ we introduced before Theorem \ref{theoreductiondebase} for a shorter one $\chi_{i}^{j}$. 
\end{nota}

\subsection{Primary matrices}

\begin{defi}
Let $\rho,N\in\N$.
 An operator $T$ on $\R^N$ is said to be primary of order $\rho$ when $T=\oplus_{i=1}^{\rho}\mathcal{A}_i$ with $\mathcal{A}_i=1$ or $R_{\theta_i}$ with $\theta_i\in\R$.
\end{defi}

\begin{rem}
One can see at first glance that if $T$ is primary of order $\rho$ on $\R^N$, then $\rho\in\llbracket\left\lfloor\frac{N+1}{2}\right\rfloor,N\rrbracket$. Moreover, $\rho$ is the relative size of $T$.
\end{rem}

We begin our study with primary matrices. However, even if the next result is a partial generalisation of Proposition \ref{propnrotpasn-1}, their proofs are independent. Moreover, this proof puts forward some useful ideas.

\begin{prop}\label{propmatprimk}
 Let $\rho\in\N$.
There is no $(\rho-1)$-supercyclic primary matrix of order $\rho$ .
\end{prop}

\begin{proof}
Let $T=\oplus_{i=1}^{\rho} \mathcal{A}_i$ be a primary matrix of order $\rho$.
Following the notations we introduced before, $\rho_i=1$ for every $1\leq i\leq \rho$. Now suppose, in order to obtain a contradiction, that $T$ is $(\rho-1)$-supercyclic.
Let $M=\Span\{x^1,\ldots,x^{\rho-1}\}$ be a $(\rho-1)$-supercyclic subspace for $T$ and then reduce the basis of $M$ with Theorem \ref{theoreductiondebase}.
First, it is worth noting that for any $p<\rho$, $\kappa_p\neq \kappa_{p+1}$.
\medskip

Assume that the contrary holds and let $p<\rho$ be the smallest integer such that $\kappa_p=\kappa_{p+1}$. This implies $\dim(\Span\{\Lambda_p\})=0$ and thus for any $\kappa_p\leq j\leq \rho-1$, $\Vert \chi_{\rho-p}^{j}\Vert=0$.
But, for all $i\in\N$ and all real sequence $(\lambda_j)_{1\leq j\leq \rho-1}$, 
$$T^i\left(\sum_{j=1}^{\rho-1}\lambda_j x^j\right)=\begin{cases}
                                      \mathcal{A}_{1}^{i}\left(\sum_{j=1}^{\rho-1}\lambda_j \chi_{1}^{j}\right)\\
\vdots\\
\mathcal{A}_{\rho-p}^{i}\left(\sum_{j=1}^{\kappa_{p}-1}\lambda_j \chi_{\rho-p}^{j}\right)\\
\mathcal{A}_{\rho-p+1}^{i}\left(\sum_{j=1}^{\kappa_{p}-1}\lambda_j \chi_{\rho-p+1}^{j}\right)\\
\vdots\\
\mathcal{A}_{\rho}^{i}\left(\sum_{j=1}^{\kappa_{1}-1}\lambda_j \chi_{\rho}^{j}\right)\\
                                     \end{cases}\begin{array}{c}
(L_1)\\
                                                 \vdots\\
(L_{\rho-p})\\
(L_{\rho-p+1})\\
\vdots\\
(L_{\rho})\\
                                                \end{array}
$$
Clearly, if $p=0$ then $(L_\rho)=0$ and $M$ fails to be $(\rho-1)$-supercyclic for $T$.
In the following, we may assume $p>0$ without loss of generality.
Then, by $(\rho-1)$-supercyclicity of $M$, there exist $(n_i)_{i\in\N}$ and $\rho-1$ real sequences $(\lambda_{1}^{(n_i)})_{i\in\N},\ldots,(\lambda_{\rho-1}^{(n_i)})_{i\in\N}$ such that for any $j\in\llbracket1,\rho\rrbracket\setminus\{\rho-p\}$,$$(L_j)\underset{i\to+\infty}{\longrightarrow}0\text{ and } (L_{\rho-p})\underset{i\to+\infty}{\longrightarrow}Y\text{ with }\Vert Y\Vert=1.$$
\medskip

We shall prove that for any $1\leq j< \kappa_p$, $\lambda_{j}^{(n_i)}\underset{i\to+\infty}{\longrightarrow}0$. Such an integer $j$ belongs to a unique interval $[\kappa_q,\kappa_{q+1}[$ and we shall prove this property by induction on $q$.
\\If $1\leq j< \kappa_1$, then $\{\chi^{j}_{\rho}\}_{1\leq j< \kappa_1}\neq\emptyset$ is a linearly independent family and $\mathcal{A}_\rho$ being an isometry, $(L_\rho)$ gives : $\sum_{j=1}^{\kappa_{1}-1}\lambda_{j}^{(n_i)} \chi_{\rho}^{j}\underset{i\to+\infty}{\longrightarrow}0$, hence $\lambda_{j}^{(n_i)}\underset{i\to+\infty}{\longrightarrow}0$ for every $1\leq j< \kappa_1$.\\
We assume that the induction hypothesis is true for $1\leq q<p$. We have to prove it for $q+1$ too.
Since $(L_{\rho-q})$ converges to $0$ and $\mathcal{A}_{\rho-q}$ being an isometry, we have: $\sum_{j=1}^{\kappa_{q+1}-1}\lambda_{j}^{(n_i)} \chi_{\rho-q}^{j}\underset{i\to+\infty}{\longrightarrow}0$.
The recurrence hypothesis implies $\lambda_{j}^{(n_i)}\underset{i\to+\infty}{\longrightarrow}0$ for any $1\leq j<\kappa_q$.
Hence $\sum_{j=\kappa_q}^{\kappa_{q+1}-1}\lambda_{j}^{(n_i)} \chi_{\rho-q}^{j}\underset{i\to+\infty}{\longrightarrow}0$.
However, $\{\chi^{j}_{\rho-q}\}_{\kappa_q\leq j< \kappa_{q+1}}\neq\emptyset$ is a linearly independent family by the reduction properties and because $\kappa_{q}\neq \kappa_{q+1}$, so $\lambda_{j}^{(n_i)}\underset{i\to+\infty}{\longrightarrow}0$ for every $1\leq j< \kappa_{q+1}$.
This ends the induction step.
\medskip

Thus, for any $1\leq j<\kappa_p$, $\lambda_{j}^{(n_i)}\underset{i\to+\infty}{\longrightarrow}0$. Considering these limits in $(L_{\rho-p})$ and the fact that $\mathcal{A}_{\rho-p}$ is an isometry yields : 
$$\mathcal{A}_{\rho-p}^{i}\left(\sum_{j=1}^{\kappa_{p}-1}\lambda_{j}^{(n_i)} \chi_{\rho-p}^{j}\right)\underset{i\to+\infty}{\longrightarrow}0.$$
But this contradicts the convergence of $(L_{\rho-p})$ to some unit vector. Hence $\kappa_p\neq \kappa_{p+1}$ for every $p<\rho$.
\medskip

Considering that $\kappa_0=1$ and that the sequence $(\kappa_p)_{0\leq p\leq \rho}$ is increasing, one obtains $\kappa_{\rho-1}\geq \rho$, hence by Theorem \ref{theoreductiondebase} $\kappa_\rho=\kappa_{\rho-1}$. This contradiction proves that $T$ is not $(\rho-1)$-supercyclic.
\end{proof}

\subsection{For a single real Jordan block}
The aim of this section is to generalise Proposition \ref{propjordpas2} to the case of a real Jordan block of arbitrary dimension. The two following lemmas are useful to express in another way the iterates of a subspace by a real Jordan block.

\begin{lem}\label{lemcomb}
Define $\Delta_{n}(i):=\binom{i}{n}-\sum_{k=1}^{n-1}\Delta_{k}(i)\binom{i}{n-k}$ for any $n\geq0$ and $i\geq0$. Then, $\Delta_{n}$ is a polynomial in $i$ of degree $n$ and its leading coefficient is $\frac{(-1)^{n+1}}{n!}$.
\end{lem}
 
\begin{proof}
We prove it by induction on $n\geq0$.
The lemma is obviously true for $n=1$ because $\Delta_1(i)=i$.\\
Assume that we have verified the induction hypothesis for $1\leq k<n$. Let us prove it for $k=n$.
Denote by $\delta_{n}$ the leading coefficient of $\Delta_{n}$.
The leading coefficient in $i$ of $\binom{i}{k}$ is $\frac{1}{k!}$ for all $k\in\Z_{+}$. Combining with the induction hypothesis, one gets:
$$\delta_n=\frac{1}{n!}-\sum_{k=1}^{n-1}\frac{(-1)^{k+1}}{(n-k)!k!}=\frac{1}{n!}\left(1-\sum_{k=1}^{n-1}(-1)^{k+1}\binom{n}{k}\right).$$

Now, it is easy to check that: $$\sum_{k=1}^{n-1}\binom{n}{k}(-1)^{k+1}=\begin{cases}
                                                                                                    0\text{ if }n\text{ is odd,}\\
2\text{ if }n\text{ is even.}
                                                                                                   \end{cases}$$
This yields: $$\delta_n=\begin{cases}
                               \frac{1}{n!}\text{ if }n\text{ is odd,}\\
-\frac{1}{n!}\text{ if }n\text{ is even.}\\
                                                                                                   \end{cases}.$$
This ends the induction and the proof of the lemma.
\end{proof}

In order to fully understand the interest of introducing the sequence $\Delta_{n}$, we also need the following lemma:
\begin{lem}\label{lemprecomb}
 Let $i,n\in\N$ with $i\geq n$ and let $(u_k)_{1\leq k\leq n}$ be a sequence of real numbers. For every $1\leq k\leq n$ define $L_k:=\sum_{j=0}^{n-k}\binom{i}{j}u_{k+j}$. Then, $L_k=u_k+\sum_{j=1}^{n-k}\Delta_{j}(i)L_{k+j}$.
\end{lem}

\begin{proof}
 
Once again, we prove this result by induction on $n$.
For $n=1$, the result is straightforward since $L_1=\sum_{j=0}^{1-1}\binom{i}{j}u_{1+j}=u_1$.\\
Now, assume that the induction hypothesis is true for any natural number strictly smaller than $n$ and let us prove it for $n$.
For $1\leq k<n$, set $\mathcal{L}_{k}=\sum_{j=0}^{n-k-1}\binom{i}{j}u_{k+j}=L_k-\binom{i}{n-k}u_n$. Then the induction hypothesis gives:
$$\mathcal{L}_{k}=u_k+\sum_{j=1}^{n-k-1}\Delta_{j}(i)\mathcal{L}_{k+j}$$
 and so for $1\leq k\leq n-1$: 
$$L_k=\mathcal{L}_{k}+\binom{i}{n-k}u_n=u_k+\sum_{j=1}^{n-k-1}\Delta_{j}(i)\mathcal{L}_{k+j}+\binom{i}{n-k}u_n.$$
Finally, use the definition of $\mathcal{L}_k$ and $\Delta_{n-k}(i)$ and note that $L_{n}=u_n$, we obtain:
\begin{align*}
L_k&=u_k+\sum_{j=1}^{n-k-1}\Delta_{j}(i)\left(L_{k+j}-\binom{i}{n-k-j}u_n\right)+\binom{i}{n-k}u_n&\\
&=u_k+\sum_{j=1}^{n-k-1}\Delta_{j}(i)L_{k+j}+\left(\binom{i}{n-k}-\sum_{j=1}^{n-k-1}\Delta_{j}(i)\binom{i}{n-k-j}\right)u_n&\\
&=u_k+\sum_{j=1}^{n-k-1}\Delta_{j}(i)L_{k+j}+\Delta_{n-k}(i)u_n&\\
&=u_k+\sum_{j=1}^{n-k}\Delta_{j}(i)L_{k+j}&
\end{align*}
This completes the proof of this lemma.
\end{proof}

Here comes now the generalisation of Proposition \ref{propjordpas2}, its proof is of significant importance to an understanding of the mechanisms involved in later proofs.
\begin{prop}\label{propblocjordr}
Let $N>1$ and $\theta\in\R$, then $J_N:=\left(\begin{array}{cccccc}
R_{\theta}&R_{\theta}&0&\cdots&0\\
0&R_{\theta}&R_{\theta}&0\cdots&0\\
0&0&\ddots&\ddots&\vdots\\
\vdots&&\ddots&\ddots&R_{\theta}\\
0&\cdots&0&0&R_{\theta}\\
\end{array}
\right)$ is not $N$-supercyclic on $\R^{2N}$ .
\end{prop}

\begin{proof}
As we already noticed, Proposition \ref{propjordpas2} proves the case $N=2$. Let then $N\geq3$ and assume to the contrary that $J_N$ is $N$-supercyclic.
Let also $M=\Span\{x^{1},\ldots,x^{N}\}$ be a $N$-supercyclic subspace which basis $x^1,\ldots,x^N$ is reduced.
Then $\kappa_N=N+1$ as we already pointed out in Theorem \ref{theoreductiondebase}.
Moreover, Proposition \ref{propjordpas2} claims that $J_2$ is not 2-supercyclic, thus the three following vectors $\small\left(\begin{array}{c}\chi_{N-1}^{1}\\\chi_{N}^{1}\\\end{array}\right),\left(\begin{array}{c}\chi_{N-1}^{2}\\\chi_{N}^{2}\\\end{array}\right),\left(\begin{array}{c}\chi_{N-1}^{3}\\\chi_{N}^{3}\\\end{array}\right)$ span a subspace of dimension 3, yielding $\kappa_2\geq4$.
In addition, $N$-supercyclicity of $M$ implies the existence of a sequence of natural numbers $(n_i)_{i\in\N}$ and $\left((\lambda_{1}^{(n_i)})_{i\in\N},\ldots,(\lambda_{N}^{(n_i)})_{i\in\N}\right)\in\left(\R^{\N}\right)^{N}$ such that:
 $$T^{n_i}\left(\sum_{j=1}^{N}\lambda_{j}^{(n_i)}x^j\right)=\left(\begin{array}{c}
R_{\theta}^{n_i}\sum_{k=0}^{N-1} \binom{n_i}{k} \left(\sum_{j=1}^{N}\lambda_{j}^{(n_i)}\left(\begin{array}{c}x_{2k+1}^{j}\\x_{2(k+1)}^{j}\\\end{array}\right)
\right)\\
R_{\theta}^{n_i}\sum_{k=0}^{N-2} \binom{n_i}{k} \left(\sum_{j=1}^{N}\lambda_{j}^{(n_i)}\left(\begin{array}{c}x_{2(k+1)+1}^{j}\\x_{2(k+2)}^{j}\\\end{array}\right)
\right)\\
\vdots\\
R_{\theta}^{n_i}\sum_{j=1}^{N}\lambda_{j}^{(n_i)}\left(\begin{array}{c}x_{2N-1}^{j}\\x_{2N}^{j}\\\end{array}\right)\\
\end{array}\right)\underset{i\to+\infty}{\longrightarrow}\left(\begin{array}{c} \left(\begin{array}{c}0\\0\\\end{array}\right)\\
\vdots\\
 \left(\begin{array}{c}0\\1\\\end{array}\right)\\\end{array}\right).$$
Denote by $(L_1),\ldots,(L_N)$ the lines appearing in $T^{n_i}\left(\sum_{j=1}^{N}\lambda_{j}^{(n_i)}x^j\right)$ above, and define $(u_k)=\sum_{j=1}^{N}\lambda_{j}^{(n_i)}\left(\begin{array}{c}x_{2k-1}^{j}\\x_{2k}^{j}\\\end{array}\right)$. Remark that the $(L_k)$'s and $(u_k)$'s depend on $i$.
Then, using Lemma \ref{lemprecomb} the preceding identity implies:
\begin{equation}\label{iter0}
\begin{cases}
\Vert(u_1)+\sum_{j=1}^{N-1}\Delta_j(n_i) (L_{j+1})\Vert&\underset{i\to+\infty}{\longrightarrow}0\\
&\ \ \ \vdots\\
\Vert(u_k)+\sum_{j=1}^{N-k}\Delta_j(n_i) (L_{j+k})\Vert&\underset{i\to+\infty}{\longrightarrow}0\\
&\ \ \ \vdots\\
\Vert(u_N)\Vert&\underset{i\to+\infty}{\longrightarrow}1\\
\end{cases}
\end{equation}
where $\Delta_j$ is defined in Lemma \ref{lemcomb}.
\medskip

We come now to the key point of the proof: we prove by induction on $k$ that $\frac{\lambda_{j}^{(n_i)}}{\Delta_{k}(n_i)}\underset{i\to+\infty}{\longrightarrow}0$, for every $1\leq k\leq N-1$ and every $1\leq j\leq \kappa_k-1$.\\
If $k=1$, then divide $(L_N)$ by $\Delta_{1}(n_i)$ and take the limit:
$$\left\Vert\sum_{j=\kappa_0}^{\kappa_1-1}\frac{\lambda_{j}^{(n_i)}}{\Delta_{1}(n_i)}\chi_{N}^{j}\right\Vert\underset{i\to+\infty}{\longrightarrow}0.$$
In addition the fact that $\{\chi_{N}^{j}\}_{\kappa_0\leq j\leq \kappa_1-1}$ is linearly independent (but not empty since $\kappa_2\geq4$) leads to: $$\frac{\lambda_{j}^{(n_i)}}{\Delta_{1}(n_i)}\underset{i\to+\infty}{\longrightarrow}0\text{ for any }\kappa_0\leq j\leq \kappa_1-1.$$
We now assume that the induction hypothesis is true for any natural number smaller than $k$ and we prove it for $k+1$.
First, divide $(L_{N-k})$ by $\Delta_{k+1}(n_i)$ and take the limit:
$$\left\Vert \sum_{j=\kappa_0}^{\kappa_k-1}\frac{\lambda_{j}^{(n_i)}}{\Delta_{k+1}(n_i)}\chi_{N-k}^{j}+\sum_{j=\kappa_k}^{\kappa_{k+1}-1}\frac{\lambda_{j}^{(n_i)}}{\Delta_{k+1}(n_i)}\chi_{N-k}^{j}+\sum_{j=1}^{k}\frac{\Delta_{j}(n_i)}{\Delta_{k+1}(n_i)}(L_{N-k+j})\right\Vert\underset{i\to+\infty}{\longrightarrow}0.$$
The induction hypothesis provides that the sum on the left of the preceding line converges to 0. Moreover, as the sequence given by the $j$-th line $(L_j)$ is bounded for all $1\leq j\leq N$ and as Lemma \ref{lemcomb} gives $\deg(\Delta_{k+1})>\deg(\Delta_{j})$ for every $1\leq j\leq k$, then the sum on the right of the previous line converges also to 0 providing:
$$\left\Vert\sum_{j=\kappa_k}^{\kappa_{k+1}-1}\frac{\lambda_{j}^{(n_i)}}{\Delta_{k+1}(n_i)}\chi_{N-k}^{j}\right\Vert\underset{i\to+\infty}{\longrightarrow}0.$$
Moreover, $\{\chi_{N-k}^{j}\}_{\kappa_k\leq j\leq \kappa_{k+1}-1}$ is either linearly independent or empty, and taking this into account in the line above and using Lemma \ref{lemcomb}, we conclude that:
$$\frac{\lambda_{j}^{(n_i)}}{\Delta_{k+1}(n_i)}\underset{i\to+\infty}{\longrightarrow}0\text{ for any }\kappa_k\leq j\leq \kappa_{k+1}-1.$$
So by induction hypothesis and as $\deg(\Delta_{k+1})>\deg(\Delta_{j})$ for every $1\leq j\leq k$, then for every $1\leq j\leq \kappa_{k+1}-1$, $\frac{\lambda_{j}^{(n_i)}}{\Delta_{k+1}(n_i)}\underset{i\to+\infty}{\longrightarrow}0$.
\medskip

We now come back to the proof of the proposition. As we claimed before, notice that $\kappa_2\geq4$ and $\kappa_{N}=N+1$. It follows that there exists $2\leq p\leq N-1$ such that $\kappa_p=\kappa_{p+1}$.
Divide then $(L_{N-p})$ by $\Delta_{p}(n_i)$ to get:
$$\left\Vert \sum_{j=\kappa_0}^{\kappa_{p+1}-1}\frac{\lambda_{j}^{(n_i)}}{\Delta_{p}(n_i)}\chi_{N-k}^{j}+\sum_{j=1}^{p}\frac{\Delta_{j}(n_i)}{\Delta_{p}(n_i)}(L_{N-p+j})\right\Vert\underset{i\to+\infty}{\longrightarrow}0.$$
The sum on the left above tends to 0 because $\frac{\lambda_{j}^{(n_i)}}{\Delta_{p}(n_i)}\underset{i\to+\infty}{\longrightarrow}0$ for any $1\leq j\leq \kappa_p-1$. Furthermore, combine Lemma \ref{lemcomb} and the boundedness of the sequence given by the $k$-th line $(L_k)$ to deal with the second sum:
$$\left\Vert\frac{(-1)^{p+1}}{p!}(L_{N})\right\Vert\underset{i\to+\infty}{\longrightarrow}0.$$
This contradicts the convergence $\Vert(L_N)\Vert\underset{i\to+\infty}{\longrightarrow}1$ given in (\ref{iter0}). So, $J_N$ is not $N$-supercyclic.

\end{proof}

\subsection{Sum of Jordan blocks with different moduli}

Later, we will need to be able to distinguish the behaviour of blocks with different moduli. The main idea is that all the coefficients we introduced in older blocks do not have a significant influence in the new block.
The following lemma is a technical tool towards this idea.

\begin{lem}\label{lemconrec}
Let $h\in\Z_{+}$ and $\gamma,m,N\in\N$ with $h<m$. Let also $T=a\mathcal{C}$ be an operator on $\R^N$ with $0<\vert a\vert<1$ and where $\mathcal{C}=\oplus_{i=1}^{\gamma}\mathcal{B}_i$, $\mathcal{B}_i$ being a Jordan block of modulus 1 with $\mathcal{A}_i=1$ or $R_{\theta_i}$.
Let $M$ be an $(m-h)$-dimensional subspace and $x^1,\ldots,x^m\in\R^N$ where $x^{h+1},\ldots,x^m$ denotes a reduced basis of $M$ (adapted to $T$ with Theorem \ref{theoreductiondebase}), and let also $0<\vert a\vert<\vert b\vert\leq1$.
Assume that there exist a strictly increasing sequence $(n_i)_{i\in\N}$ and $m$ real sequences $(\lambda_{1}^{(n_i)})_{i\in\N},\ldots,(\lambda_{m}^{(n_i)})_{i\in\N}$ and $q\in\Z_{+}$ such that $\frac{b^{n_i} \lambda_{j}^{(n_i)}}{n_{i}^{q}}\underset{i\to+\infty}{\longrightarrow}0$ for any $1\leq j\leq h$ and $T^{n_i}(\sum_{j=1}^{m}\lambda_{j}^{(n_i)}x^j)\underset{i\to+\infty}{\longrightarrow}0$.
Then, there exists $q'\in\Z_{+}$ satisfying $\frac{a^{n_i}\lambda_{j}^{(n_i)}}{n_{i}^{q'}}\underset{i\to+\infty}{\longrightarrow}0$ for any $1\leq j\leq m$.
\end{lem}

\begin{proof}
Denote as usual $\tau_i \rho_i$ the size of the block $\mathcal{B}_i$ with $\tau_i=1,2$ and $\rho=\sum_{i=1}^{\gamma}\rho_i$. Then reducing $T$ and keeping in mind that the last $(m-h)$ vectors from $\{x_{1},\ldots,x_{m}\}$ are reduced, one may obtain the following equality

$${\scriptstyle T^{n_i}\left(\sum_{j=1}^{m}\lambda_{j}^{(n_i)}x^j\right)=\begin{cases}
                                      a^{n_i}\mathcal{A}_{1}^{n_i}\left(\sum_{j=1}^{m}\lambda_{j}^{(n_i)}\chi_{1}^{j}+\sum_{j=1}^{\rho_1-1}\binom{n_i}{j}\sum_{g=1}^{h+\kappa_{\rho-j}-1}\lambda_{g}^{(n_i)}\chi^{g}_{1+j}\right)\\
\vdots\\
a^{n_i}\mathcal{A}_{1}^{n_i}\left(\sum_{j=1}^{h+\kappa_{\rho+1-\rho_1}-1}\lambda_{j}^{(n_i)}\chi_{\rho_1}^{j}\right)\\
\vdots\\
a^{n_i}\mathcal{A}_{\gamma}^{n_i}\left(\sum_{j=1}^{h+\kappa_{\rho_\gamma}-1}\lambda_{j}^{(n_i)}\chi_{\rho+1-\rho_\gamma}^{j}+\sum_{j=1}^{\rho_\gamma-1}\binom{n_i}{j}\sum_{g=1}^{h+\kappa_{\rho_\gamma-j}-1}\lambda_{g}^{(n_i)}\chi^{g}_{\rho-\rho_\gamma+1+j}\right)\\
\vdots\\
a^{n_i}\mathcal{A}_{\gamma}^{n_i}\left(\sum_{j=1}^{h+\kappa_{1}-1}\lambda_{j}^{(n_i)} \chi_{\rho}^{j}\right)\\
                                     \end{cases}\begin{array}{c}
\left(L_1\right)\\
                                                 \vdots\\
\left(L_{\rho_1}\right)\\
\vdots\\
\left(L_{\rho-\rho_\gamma+1}\right)\\
\vdots\\
\left(L_{\rho}\right)\\
                                                \end{array}}$$

We are going to make a constant difference between the $h$ first vectors and the $m-h$ last vectors and one has to keep in mind that the common notations for a reduced basis only refers to a reduction on the last vectors.\\
\medskip

Let us prove the lemma by decreasing induction on $l\in\{1,\ldots,\rho\}$. 
Our induction hypothesis is that for every $1\leq l\leq \rho$, there exists $q'\in\Z_{+}$ such that for any $1\leq j\leq h+\kappa_{\rho+1-l}-1$, $\frac{a^{n_i}\lambda_{j}^{(n_i)}}{n_{i}^{q'}}\underset{i\to+\infty}{\longrightarrow}0$.

We begin by proving the induction hypothesis with $l=\rho$.
Observing that $\mathcal{A}_\gamma$ is an isometry, then $(L_\rho)$ gives:
\begin{equation}\label{l_r}
 a^{n_i}\left(\sum_{j=1}^{h}\lambda_{j}^{(n_i)} \chi_{\rho}^{j}+\sum_{j=h+\kappa_0}^{h+\kappa_{1}-1}\lambda_{j}^{(n_i)} \chi_{\rho}^{j}\right)\underset{i\to+\infty}{\longrightarrow}0.
\end{equation}
According to the assumptions, for any $1\leq j\leq h$, $\frac{b^{n_i} \lambda_{j}^{(n_i)}}{n_{i}^{q}}\underset{i\to+\infty}{\longrightarrow}0$. Recall that $0<\vert a\vert<\vert b\vert\leq1$, we deduce $a^{n_i} \lambda_{j}^{(n_i)}\underset{i\to+\infty}{\longrightarrow}0$ for any $1\leq j\leq h$.
Substitute this result into (\ref{l_r}): 
$$a^{n_i}\sum_{j=h+\kappa_0}^{h+\kappa_{1}-1}\lambda_{j}^{(n_i)} \chi_{\rho}^{j}\underset{i\to+\infty}{\longrightarrow}0.$$
Moreover, $\{\chi_{\rho}^{j}\}_{h+\kappa_0}^{h+\kappa_1-1}$ is linearly independent or empty so for every $1\leq j\leq h+\kappa_1-1$, $a^{n_i}\lambda_{j}^{(n_i)}\underset{i\to+\infty}{\longrightarrow}0$.

Assume that $l\in\{1,\ldots,\rho\}$ and that the induction hypothesis is true for natural numbers strictly greater than $l$ and smaller than $\rho$.
By induction hypothesis, there exists $q'\in\Z_{+}$ so that: $$\frac{a^{n_i}\lambda_{j}^{(n_i)}}{n_{i}^{q'}}\underset{i\to+\infty}{\longrightarrow}0\text{ for any }1\leq j\leq h+\kappa_{\rho-l}-1.$$
\\
If $\kappa_{\rho-l}<\kappa_{\rho+1-l}$, then $(L_l)$ gives: 
\begin{equation}\label{eqjesaispas}
 a^{n_i}\mathcal{A}_{f}^{n_i}\left(\sum_{j=1}^{h+\kappa_{\rho-l}-1}\lambda_{j}^{(n_i)} \chi_{l}^{j}+\sum_{j=h+\kappa_{\rho-l}}^{h+\kappa_{\rho+1-l}-1}\lambda_{j}^{(n_i)} \chi_{l}^{j}+\sum_{j=1}^{d}\binom{n_i}{j}\sum_{g=1}^{h+\kappa_{\rho+1-l-j}-1}\lambda_{g}^{(n_i)}\chi^{g}_{l+j}\right)\underset{i\to+\infty}{\longrightarrow}0.
\end{equation}
with $f\in\llbracket1,\gamma\rrbracket$ and $d\in\llbracket0,\rho_f-1\rrbracket$.
Note that $\mathcal{A}_f$ is an isometry and divide the preceding equation by $n_{i}^{q'}$ then the first sum tends to 0 by induction hypothesis: 
\begin{equation*}
\frac{1}{n_{i}^{q'}}a^{n_i}\left(\sum_{j=h+\kappa_{\rho-l}}^{h+\kappa_{\rho+1-l}-1}\lambda_{j}^{(n_i)} \chi_{l}^{j}+\sum_{j=1}^{d}\binom{n_i}{j}\sum_{g=1}^{h+\kappa_{\rho+1-l-j}-1}\lambda_{g}^{(n_i)}\chi^{g}_{l+j}\right)\underset{i\to+\infty}{\longrightarrow}0.
\end{equation*}
Now divide the last equation by $n_{i}^{d}$ then combining the boundedness of $\left(\frac{\binom{n_i}{j}}{n_{i}^d}\right)_{i\in\N}$ for any $1\leq j\leq d$ and the induction hypothesis for the last sum, we obtain:
\begin{equation*}
\frac{1}{n_{i}^{q'+d}}a^{n_i}\left(\sum_{j=h+\kappa_{\rho-l}}^{h+\kappa_{\rho+1-l}-1}\lambda_{j}^{(n_i)} \chi_{l}^{j}\right)\underset{i\to+\infty}{\longrightarrow}0
\end{equation*}
Hence, as $\kappa_{\rho-l}<\kappa_{\rho+1-l}$, then $\{\chi_{l}^{j}\}_{j=h+\kappa_{\rho-l}}^{j=h+\kappa_{\rho+1-l}-1}$ is linearly independent, and therefore:
$$\frac{a^{n_i}\lambda_{j}^{(n_i)}}{n_{i}^{q'+d}}\underset{i\to+\infty}{\longrightarrow}0\text{ for any }1\leq j\leq h+\kappa_{\rho+1-l}-1.$$
\\
If $\kappa_{\rho-l}=\kappa_{\rho+1-l}$, then the proof is the same that the previous one but the first sum is missing in (\ref{eqjesaispas}).
\\Hence there exists $q'\in\Z_{+}$ such that for all $j\in\{1,\ldots,h+\kappa_{\rho}-1\}$, $\frac{a^{n_{i}}\lambda_{j}^{(n_{i})}}{n_{i}^{q'}}\underset{i\to+\infty}{\longrightarrow}0$.
This ends the induction and also the proof of the lemma because Theorem \ref{theoreductiondebase} implies $\kappa_\rho=m-h+1$, thus $h+\kappa_{\rho}-1=m$.
\end{proof}

\subsection{Sum of Jordan blocks with the same modulus}

This lemma deals with the growth of coefficients as we did before but in the case of Jordan blocks with the same modulus. It actually depends on the relative size of the two biggest blocks.
The proof is close to the one of Lemma \ref{lemconrec} but is more technical.
\begin{lem}\label{lemconrec2}
Let $\gamma,m,N\in\N$, $\gamma\geq2$. Let also $T=\oplus_{i=1}^{\gamma}\mathcal{B}_i$ be an operator on $\R^N$, where $\mathcal{B}_i$ is a Jordan block of modulus one and $\mathcal{A}_i=1$ or $R_{\theta_i}$. Moreover, assume that one of the following conditions holds:
\begin{align}
 &\rho_1\geq \ldots\geq \rho_\gamma&\label{H1}\\
 &\rho_2\geq \ldots\geq \rho_\gamma \text{ and } \rho_1=\rho_2-1&\label{H2}.
\end{align}
Let $M$ be an $m$-dimensional subspace and let $x^1,\ldots,x^m\in\R^N$ denote a reduced basis of $M$ with Theorem \ref{theoreductiondebase} for which there exist a strictly increasing sequence of natural numbers $(n_i)_{i\in\N}$ and real sequences $(\lambda_{j}^{(n_i)})_{i\in\N,1\leq j\leq m}$ such that $T^{n_i}(\sum_{j=1}^{m}\lambda_{j}^{(n_i)}x^j)\underset{i\to+\infty}{\longrightarrow}(\underbrace{1,\ldots,1}_{\rho_1\text{ times}},0,\ldots,0)$.
\\Then $\frac{\lambda_{j}^{(n_i)}}{n_{i}^{\rho_1}}\underset{i\to+\infty}{\longrightarrow}0$ for any $1\leq j\leq m$.
\end{lem}

\begin{proof}
Denote by $\tau_i \rho_i$ the size of the block $\mathcal{B}_i$ with $\tau_i=1$ or $2$ and $\rho=\sum_{i=1}^{\gamma}\rho_i$. Then,
$$T^{n_i}\left(\sum_{j=1}^{m}\lambda_{j}^{(n_i)}x^j\right)=\begin{cases}
                                      \mathcal{A}_{1}^{n_i}\left(\sum_{j=1}^{m}\lambda_{j}^{(n_i)}\chi_{1}^{j}+\sum_{j=1}^{\rho_1-1}\binom{n_i}{j}\sum_{g=1}^{\kappa_{\rho-j}-1}\lambda_{g}^{(n_i)}\chi^{g}_{1+j}\right)\\
\vdots\\
\mathcal{A}_{1}^{n_i}\left(\sum_{j=1}^{\kappa_{\rho+1-\rho_1}-1}\lambda_{j}^{(n_i)}\chi_{\rho_1}^{j}\right)\\
\vdots\\
\mathcal{A}_{\gamma}^{n_i}\left(\sum_{j=1}^{\kappa_{\rho_\gamma}-1}\lambda_{j}^{(n_i)}\chi_{\rho+1-\rho_\gamma}^{j}+\sum_{j=1}^{\rho_\gamma-1}\binom{n_i}{j}\sum_{g=1}^{\kappa_{\rho_\gamma-j}-1}\lambda_{g}^{(n_i)}\chi^{g}_{\rho-\rho_\gamma+1+j}\right)\\
\vdots\\
\mathcal{A}_{\gamma}^{n_i}\left(\sum_{j=1}^{\kappa_{1}-1}\lambda_{j}^{(n_i)} \chi_{\rho}^{j}\right)\\
                                     \end{cases}\begin{array}{c}
(L_1)\\
                                                 \vdots\\
(L_{\rho_1})\\
\vdots\\
(L_{\rho-\rho_\gamma+1})\\
\vdots\\
(L_{\rho})\\
                                                \end{array}$$
becomes
$$T^{n_i}\left(\sum_{j=1}^{m}\lambda_{j}^{(n_i)}x^j\right)=\begin{cases}
                                      \mathcal{A}_{1}^{n_i}\left(\sum_{j=1}^{m}\lambda_{j}^{(n_i)}\chi_{1}^{j}+\sum_{j=1}^{\rho_1-1}\Delta_{j}(n_i)(L_{1+j})\right)\\
\vdots\\
\mathcal{A}_{1}^{n_i}\left(\sum_{j=1}^{\kappa_{\rho+1-\rho_1}-1}\lambda_{j}^{(n_i)}\chi_{\rho_1}^{j}\right)\\
\vdots\\
\mathcal{A}_{\gamma}^{n_i}\left(\sum_{j=1}^{\kappa_{\rho_\gamma}-1}\lambda_{j}^{(n_i)}\chi_{\rho+1-\rho_\gamma}^{j}+\sum_{j=1}^{\rho_\gamma-1}\Delta_{j}(n_i)(L_{\rho-\rho_\gamma+1+j})\right)\\
\vdots\\
\mathcal{A}_{\gamma}^{n_i}\left(\sum_{j=1}^{\kappa_{1}-1}\lambda_{j}^{(n_i)} \chi_{\rho}^{j}\right)\\
                                     \end{cases}\begin{array}{c}
(L_1)\\
                                                 \vdots\\
(L_{\rho_1})\\
\vdots\\
(L_{\rho-\rho_\gamma+1})\\
\vdots\\
(L_{\rho})\\
                                                \end{array}$$
with the help of Lemma \ref{lemcomb} and Lemma \ref{lemprecomb}.
Moreover, remark that every $l\in\{1,\ldots,\rho\}$ can be written $l=\rho-\sum_{i=\gamma-j+1}^{\gamma}\rho_i-d_{l}$ with $j\in\llbracket0,\gamma-1\rrbracket$ and $d_{l}\in\llbracket0,\rho_{\gamma-j}-1\rrbracket$ in a unique way.\\Thus for every $l\in\{1,\ldots,\rho\}$ we can define $\delta_l=\begin{cases}d_{l}+1\text{ if }j=0\\\max(d_{l}+1,\rho_{\gamma-j+1})\text{ if }j\neq0\end{cases}$.
Roughly speaking $\delta_l$ is the relative size of the biggest Jordan block of $T$ under the $l$-th line (included) of $T$. Moreover, we need to precise that a line has to be understood as a normal line in a classical Jordan block but as two lines for a real Jordan block. 
In addition, denote $\nu_l=\begin{cases}
                \delta_l&\text{ if }(\ref{H1})\text{ holds }\\
\delta_l-1&\text{ if }(\ref{H2})\text{ holds }
               \end{cases}$.
\medskip

We are going to prove that for any $1\leq j\leq \kappa_{\rho+1-l}-1$, $\frac{\lambda_{j}^{(n_i)}}{n_{i}^{\nu_l}}\underset{i\to+\infty}{\longrightarrow}0$ by decreasing induction on $l\in\{1,\ldots,\rho\}$.\\
Consider first that $l=\rho$, $\mathcal{A}_\gamma$ being an isometry, $(L_\rho)$ gives:
\begin{equation}\label{l_r2}
\sum_{j=1}^{\kappa_{1}-1}\lambda_{j}^{(n_i)} \chi_{\rho}^{j}\underset{i\to+\infty}{\longrightarrow}0.
\end{equation}
\\By independence of the family $\{\chi_{\rho}^{j}\}_{\kappa_0}^{\kappa_1-1}$, we conclude that for any $1\leq j\leq \kappa_1-1$, $\lambda_{j}^{(n_i)}\underset{i\to+\infty}{\longrightarrow}0$.\\
Now we assume that the induction hypothesis is satisfied for $l+1,\ldots,\rho$ and we prove it for $l$.
The induction hypothesis yields:$$\frac{\lambda_{j}^{(n_i)}}{n_{i}^{\nu_{l+1}}}\underset{i\to+\infty}{\longrightarrow}0\text{, for any }1\leq j\leq \kappa_{\rho-l}-1.$$
\\
If $\kappa_{\rho-l}<\kappa_{\rho+1-l}$, then remember $(L_l)$: 
\begin{equation*}
 \left\Vert \mathcal{A}_{f}^{n_i}\left(\sum_{j=1}^{\kappa_{\rho-l}-1}\lambda_{j}^{(n_i)} \chi_{l}^{j}+\sum_{j=\kappa_{\rho-l}}^{\kappa_{\rho+1-l}-1}\lambda_{j}^{(n_i)} \chi_{l}^{j}+\sum_{j=1}^{d_{l}}\Delta_{j}(n_i)L_{l+j})\right)\right\Vert\underset{i\to+\infty}{\longrightarrow}0\text{ or }1
\end{equation*}
where $f\in\llbracket1,\gamma\rrbracket$ and $d_{l}\in\llbracket0,\rho_f-1\rrbracket$.
Keep in mind that $A_f$ is an isometry and divide the preceding line by $n_{i}^{\nu_{l+1}}$, then by induction the first sum tends to zero and we get: 
\begin{equation}\label{eqidontknow}
\frac{1}{n_{i}^{\nu_{l+1}}}\left(\sum_{j=\kappa_{\rho-l}}^{\kappa_{\rho+1-l}-1}\lambda_{j}^{(n_i)} \chi_{l}^{j}+\sum_{j=1}^{d_{l}}\Delta_{j}(n_i)(L_{l+j})\right)\underset{i\to+\infty}{\longrightarrow}0.
\end{equation}
It suffices then to use Lemma \ref{lemcomb} and to compare $\nu_{l+1}$ and $d_{l}$ to obtain:

$\bullet$ if (\ref{H1}) holds:
$$\begin{cases}
\frac{1}{n_{i}^{\nu_{l+1}}}\left(\sum_{j=\kappa_{\rho-l}}^{\kappa_{\rho+1-l}-1}\lambda_{j}^{(n_i)} \chi_{l}^{j}\right)\underset{i\to+\infty}{\longrightarrow}0&\text{ if }d_{l}<\nu_{l+1}\text{ i.e. }\delta_l=\delta_{l+1}.\\
\frac{1}{n_{i}^{\nu_{l+1}}}\left(\sum_{j=\kappa_{\rho-l}}^{\kappa_{\rho+1-l}-1}\lambda_{j}^{(n_i)} \chi_{l}^{j}\right)+\frac{(-1)^{d_{l}+1}}{d_{l}!}(L_{l+d_{l}})\underset{i\to+\infty}{\longrightarrow}0&\text{ if }d_{l}=\nu_{l+1}\text{ i.e. }\delta_l=\delta_{l+1}+1.\\
\end{cases}$$
Because of $\kappa_{\rho-l}<\kappa_{\rho+1-l}$, the family $\{\chi_{l}^{j}\}_{\kappa_{\rho-l}}^{\kappa_{\rho+1-l}-1}$ is linearly independent and we deduce that for any $1\leq j\leq \kappa_{\rho+1-l}-1$, $$\frac{\lambda_{j}^{(n_i)}}{n_{i}^{\delta_l}}\underset{i\to+\infty}{\longrightarrow}0.$$

$\bullet$ if (\ref{H2}) holds:
$$\begin{cases}
\frac{1}{n_{i}^{\nu_{l+1}}}\left(\sum_{j=\kappa_{\rho-l}}^{\kappa_{\rho+1-l}-1}\lambda_{j}^{(n_i)} \chi_{l}^{j}\right)\underset{i\to+\infty}{\longrightarrow}0&\text{ if }d_{l}<\nu_{l+1}\text{ hence }\delta_l=\delta_{l+1}.\\
\frac{1}{n_{i}^{\nu_{l+1}}}\left(\sum_{j=\kappa_{\rho-l}}^{\kappa_{\rho+1-l}-1}\lambda_{j}^{(n_i)} \chi_{l}^{j}\right)+\frac{(-1)^{d_{l}+1}}{d_{l}!}(L_{l+d_{l}})\underset{i\to+\infty}{\longrightarrow}0&\text{ if }d_{l}=\nu_{l+1}\text{ hence }\delta_l=\delta_{l+1}.\\
\frac{1}{n_{i}^{\nu_{l+1}+1}}\left(\sum_{j=\kappa_{\rho-l}}^{\kappa_{\rho+1-l}-1}\lambda_{j}^{(n_i)} \chi_{l}^{j}\right)+\frac{(-1)^{d_{l}+1}}{d_{l}!}(L_{l+d_{l}})\underset{i\to+\infty}{\longrightarrow}0&\text{ if }d_{l}=\nu_{l+1}+1\text{, i.e. }\delta_l=\delta_{l+1}+1.\\
\end{cases}$$
However, in the second and third cases just above, we are not working on the block $\mathcal{B}_1$ yet because the condition $\rho_1=\rho_2-1$ is not compatible with $d_{l}=\delta_{l+1}-1$ or $d_{l}=\delta_{l+1}$ given by the second and third cases. Thus, $(L_{l+d_{l}})\underset{i\to+\infty}{\longrightarrow}0$.
Moreover, $\kappa_{\rho-l}<\kappa_{\rho+1-l}$ yields to the independence of the family $\{\chi_{l}^{j}\}_{\kappa_{\rho-l}}^{\kappa_{\rho+1-l}-1}$. Hence, we deduce that for any $1\leq j\leq \kappa_{\rho+1-l}-1$, $$\frac{\lambda_{j}^{(n_i)}}{n_{i}^{\delta_{l}-1}}\underset{i\to+\infty}{\longrightarrow}0.$$
\\
If $\kappa_{\rho-l}=\kappa_{\rho+1-l}$, then the proof is the same apart from that in (\ref{eqidontknow}) the first term is missing.
\\Thus for every $1\leq j\leq \kappa_\rho-1$, $\frac{\lambda_{j}^{(n_{i})}}{n_{i}^{\nu_{1}}}.$
This ends the induction process.
\medskip

Then, by Theorem \ref{theoreductiondebase} $\kappa_{\rho}-1=m$ and noticing that $\nu_1=\rho_1$ in case (\ref{H1}) and $\nu_1=\rho_2-1$ in case (\ref{H2}), we end the proof of the lemma.
\end{proof}

From this lemma, we deduce a general result for operators which are given as a direct sum of Jordan blocks with the same moduli. 

\begin{theo}\label{theodsbloc}
Let $T$ be an operator on $\R^N$ which is a direct sum of Jordan blocs of modulus 1. Then $T$ is not $(\rho-1)$-supercyclic.
\end{theo}

\begin{proof}
Let us define the natural number $D:=\sum_{i=1}^{\gamma}(\rho_i-1)$ as the degree of $T$.
\medskip

The proof is done by induction on the degree $D$ of $T$.
If $D=0$, then $T$ is a primary matrix of order $\rho$ and Proposition \ref{propmatprimk} claims that $T$ is not $(\rho-1)$-supercyclic.\\
Assume that the induction hypothesis is true from 0 to $D-1$, let us prove it for $D$.\\
Suppose in order to obtain a contradiction that $T$ is $(\rho-1)$-supercyclic.
Without loss of generality, we can assume that $T=\oplus_{i=1}^{\gamma}\mathcal{B}_i$, where $\mathcal{B}_i$ is a Jordan block of modulus one and $\rho_{1}\geq \ldots\geq \rho_{\gamma}$.
$T$ is clearly not $(\rho-1)$-supercyclic when $T$ contains only one block due to either Proposition \ref{propblocjordr} if the block is real or \cite{Bou} if it is a classical Jordan block.
We can also assume that $\rho_{1}>1$ thanks to Proposition \ref{propmatprimk}.

Thus one can write:
$$T=\left(\begin{array}{cccccccc}
\multicolumn{1}{c|}{\mathcal{A}_{1}}   &\multicolumn{1}{|c}{\mathcal{A}_{1}}&0&\cdots&\multicolumn{1}{c|}{0}  &0&\cdots&0 \\\cline{1-5}
0 &\multicolumn{4}{|c|}{\mathcal{C}_{1}}&0&\cdots&0 \\ \cline{1-5}
0&\cdots &\cdots&0&0&\mathcal{B}_2&\ddots&0\\
\vdots&\ddots&\ddots&0&0&0&\ddots&0 \\
0 & \cdots&\cdots&0      &0&0&0&\mathcal{B}_{\gamma} \\
\end{array}\right)$$
with $\mathcal{B}_1=\left(\begin{array}{ccccc}
\multicolumn{1}{c|}{\mathcal{A}_{1}}   &\multicolumn{1}{c}{\mathcal{A}_{1}}&0&\cdots&\multicolumn{1}{c}{0} \\\cline{1-5}
0 &\multicolumn{4}{|c}{\mathcal{C}_{1}}\end{array}\right)$.
Denote also by $S$ the diagonal block matrix being the direct sum of $\mathcal{C}_1,\mathcal{B}_2,\ldots,\mathcal{B}_\gamma$.
As $T$ is $(\rho-1)$-supercyclic, denote $M=\Span\{x^1,\ldots,x^{\rho-1}\}$ a $(\rho-1)$-supercyclic subspace for $T$ where the basis is reduced with Theorem \ref{theoreductiondebase}.
According to the induction hypothesis, $S$ is not $(\rho-2)$-supercyclic, and thus there exists $p<\rho$ such that $\kappa_p=\rho$.
Indeed, suppose by way of contradiction that $\kappa_{\rho-1}<\rho$, this means that $\chi_{j}^{\rho-1}=\left(\begin{array}{c}0\\0\\\end{array}\right)$ for every $2\leq j\leq \rho$. Thus, $d:=\dim(\Span\{(\chi_{j}^{1})_{2\leq j\leq \rho},\ldots,(\chi_{j}^{\rho-1})_{2\leq j\leq \rho}\})\leq \rho-2$. Hence, as $M$ is $(\rho-1)$-supercyclic for $T$, then $S$ is $d$-supercyclic. But this would contradict the fact that $S$ is not $(\rho-2)$-supercyclic.\\
Since $T$ is $(\rho-1)$-supercyclic, there exist a strictly increasing sequence of natural numbers $(n_i)_{i\in\N}$ and real sequences $(\lambda_{j}^{(n_i)})_{i\in\N,1\leq j\leq \rho-1}$ such that:
$T^{n_i}\left(\sum_{j=1}^{\rho-1}\lambda^{(n_i)}_{j}x^j\right) \underset{i\to+\infty}{\longrightarrow}(\underbrace{1,\ldots,1}_{\rho_1\text{ times}},0,\ldots,0)$.
Then, there are two options: either $\rho_1-1\geq \rho_2$ or $\rho_1=\rho_2$.
Moreover, in both cases, Lemma \ref{lemconrec2} applied to $T=S$ leads to $\frac{\lambda^{(n_i)}_{j}}{n_{i}^{\rho_1-1}}\underset{i\to+\infty}{\longrightarrow}0$ for any $1\leq j\leq \rho-1$ because there exists $p<\rho$ such that $\kappa_p=\rho$.
On the other hand, applying Lemma \ref{lemprecomb} to the first line of $T^{n_i}\left(\sum_{j=1}^{\rho-1}\lambda^{(n_i)}_{j}x^j\right)$ and taking the limit provides:
$$\left\Vert \mathcal{A}_{1}^{n_i}\left(\sum_{j=1}^{\rho-1}\lambda_{j}^{(n_i)}\chi_{1}^{j}+\sum_{j=1}^{\rho_1-1}\Delta_{j}(n_i)(L_{1+j})\right)\right\Vert\underset{i\to+\infty}{\longrightarrow}1.$$
Now, divide this by $n_{i}^{\rho_1-1}$ and recall that $A_1$ is an isometry:
$$\sum_{j=1}^{\rho-1}\frac{\lambda_{j}^{(n_i)}}{n_{i}^{\rho_1-1}}\chi_{1}^{j}+\sum_{j=1}^{\rho_1-1}\frac{\Delta_{j}(n_i)}{n_{i}^{\rho_1-1}}(L_{1+j})\underset{i\to+\infty}{\longrightarrow}0.$$
Using both $\frac{\lambda^{(n_i)}_{j}}{n_{i}^{\rho_1-1}}\underset{i\to+\infty}{\longrightarrow}0$ and Lemma \ref{lemcomb}, one may observe that all the terms in the previous equation are tending to 0 apart from the term in $j=\rho_1-1$ in the last sum:
$$\frac{(-1)^{\rho_1}}{(\rho_1-1)!}\underset{i\to+\infty}{\longrightarrow}0.$$
This last result is obviously contradictory, thus $T$ is not $(\rho-1)$-supercyclic and this ends the induction process.
\end{proof}

\subsection{General matrix}

The next Theorem reduces the study of $m$-supercyclic operators on $\R^N$ to that of operators which are direct sums of Jordan blocks of modulus one.
\begin{theo}\label{theovpdiff}
Let $T$ be such that $T=\oplus_{i=1}^{\gamma}a_i \mathcal{C}_i$ where $\vert a_1\vert<\cdots<\vert a_\gamma\vert\leq1$, and $\mathcal{C}_i$ is a direct sum of Jordan blocks of modulus one for any $1\leq i\leq \gamma$.
Assume that for any $1\leq i\leq \gamma$, $\mathcal{C}_i$ is $m_i$-supercyclic and that $m_i$ is optimal.
Then, $T$ is not $\left((\sum_{i=1}^{\gamma}m_i)-1\right)$-supercyclic.
\end{theo}

\begin{proof}
Let $T_p:=\oplus_{i=\gamma+1-p}^{\gamma}a_i \mathcal{C}_i$ and $t(p)$ denotes this matrix's size.
We may prove by induction that for any $1\leq p\leq \gamma$, $T_p$ is not $\left((\sum_{i=\gamma+1-p}^{\gamma}m_i)-1\right)$-supercyclic.
Actually, we prove a little bit more:
\begin{align*}
&\text{For any $1\leq p\leq \gamma$, $T_p$ is not $((\sum_{i=\gamma+1-p}^{\gamma}m_i)-1)$-supercyclic.}\label{hyprec1}\\\nonumber
&\text{ Moreover, for any $b$-supercyclic subspace with reduced basis $M=\Span\{x^1,\ldots,x^b\}$ }\\\nonumber
&\text{and with $\sum_{i=\gamma+1-p}^{\gamma}m_i\leq b\leq t(p)$, if $T_{p}^{n_i}\left(\sum_{j=1}^{b}\lambda_{j}^{(n_i)}x^j\right)\underset{i\to+\infty}{\longrightarrow}0$, then there exists $q\in\Z_{+}$ so that:}\\\nonumber
&\text{$\frac{a_{\gamma+1-p}^{n_i}\lambda_{j}^{(n_i)}}{n_{i}^{q}}\underset{i\to+\infty}{\longrightarrow}0$ for any $1\leq j\leq b$.}\\\nonumber
\end{align*}
Assume that $p=1$, then
$T_p=a_\gamma \mathcal{C}_\gamma$ and by definition, $T_p$ is $m_\gamma$-supercyclic and $m_\gamma$ is the minimum supercyclic constant.
Let $M=\Span\{x^1,\ldots,x^b\}$ be a $b$-supercyclic subspace with reduced basis and $m_\gamma\leq b\leq t(1)$. Let also, $T_{p}^{n_i}\left(\sum_{j=1}^{b}\lambda_{j}^{(n_i)}x^j\right)\underset{i\to+\infty}{\longrightarrow}0$.
Applying Lemma \ref{lemconrec} with $h=0$, $\mathcal{C}=\mathcal{C}_\gamma$, $a=a_\gamma$, $m=b$, $N=t(p)$, we see that there exists $q\in\Z_{+}$ such that $\frac{a_{\gamma}^{n_i}\lambda_{j}^{(n_i)}}{n_{i}^{q}}\underset{i\to+\infty}{\longrightarrow}0$ for any $1\leq j\leq b$.
\\Assume the induction hypothesis true for integers lower than $p$, and let us prove it for $p$.
Write $T_p=a_{\gamma+1-p}\mathcal{C}_{\gamma+1-p}\oplus T_{p-1}$ and $\mathcal{C}_{\gamma+1-p}=\oplus_{i=1}^{t}\mathcal{B}_i$ where $\mathcal{B}_i$ is a Jordan block of modulus one and of size $\tau_i \rho_i$ with $\tau_i=1$ or $2$ and define $\rho:=\sum_{i=1}^{t} \rho_i$.
\medskip

In order to obtain a contradiction, assume that $k=(\sum_{i=\gamma+1-p}^{\gamma}m_i)-1$ and let $M=\Span\{x^1,\ldots,x^k\}$ be a $k$-supercyclic subspace with reduced basis given by Theorem \ref{theoreductiondebase}.
Then for any $1\leq i\leq k$, decompose $x^i=y^i\oplus z^i$ relatively to the direct sum decomposition of $T_p$ stated above. A straightforward use of the induction hypothesis provides $h:=\dim(\Span\{z^1,\ldots,z^k\})\geq \sum_{i=\gamma+2-p}^{\gamma}m_i$.
\\Furthermore, we can show that $\dim(\Span\{y^{h+1},\ldots,y^{k}\})\geq m_{\gamma+1-p}$, indeed it suffices to prove that $\Span\{y^{h+1},\ldots,y^{k}\}$ is supercyclic for $a_{\gamma+1-p} \mathcal{C}_{\gamma+1-p}$.
Thus take any $u$ belonging to the domain of $\mathcal{C}_{\gamma+1-p}$, then there exists $(n_i)_{i\in\N}$ and $(\lambda_{j}^{(n_i)})_{i\in\N,1\leq j\leq k}$ so that $T_{p}^{n_i}(\sum_{j=1}^{k}\lambda_{j}^{(n_i)}x^j)\underset{i\to+\infty}{\longrightarrow}u\oplus 0$ by $k$-supercyclicity of $T_p$.
Moreover the induction hypothesis implies that for any $1\leq j\leq h$, there exists $q\in\Z_{+}$ such that $\frac{a_{\gamma+1-(p-1)}^{n_i}\lambda_{j}^{(n_i)}}{n_{i}^{q}}\underset{i\to+\infty}{\longrightarrow}0$ for any $1\leq j\leq h$ and since $a_{\gamma+1-p}<a_{\gamma+2-p}$, we obtain:
\begin{equation}\label{equpol}
a_{\gamma+1-p}^{n_i}\lambda_{j}^{(n_i)}P(n_i)\underset{i\to+\infty}{\longrightarrow}0\text{ for any polynomial }P. 
\end{equation}
Now we come back to $T_{p}^{n_i}(\sum_{j=1}^{k}\lambda_{j}^{(n_i)}x^j)\underset{i\to+\infty}{\longrightarrow}u\oplus 0$, projecting onto the first components, and separating the sum, we get:
\begin{equation}\label{equ}
(a_{\gamma+1-p}\mathcal{C}_{\gamma+1-p})^{n_i}(\sum_{j=1}^{h}\lambda_{j}^{(n_i)}y_j)+(a_{\gamma+1-p}\mathcal{C}_{\gamma+1-p})^{n_i}(\sum_{j=h+1}^{k}\lambda_{j}^{(n_i)}y_j)\underset{i\to+\infty}{\longrightarrow}u.
\end{equation}
Then, focus on the first sum of the preceding line:
$$\scriptstyle(a_{\gamma+1-p}\mathcal{C}_{\gamma+1-p})^{n_i}(\sum_{j=1}^{h}\lambda_{j}^{(n_i)}y_j)=\begin{cases}
                                      \scriptstyle a_{\gamma+1-p}^{n_i}\mathcal{A}_{1}^{n_i}(\sum_{j=1}^{h}\lambda_{j}^{(n_i)}\chi_{1}^{j}+\sum_{j=1}^{\rho_1-1}\binom{n_i}{j}\sum_{g=1}^{h}\lambda_{g}^{(n_i)}\chi^{g}_{1+j})\\
\vdots\\
\scriptstyle a_{\gamma+1-p}^{n_i}\mathcal{A}_{1}^{n_i}\left(\sum_{j=1}^{h}\lambda_{j}^{(n_i)} \chi_{\rho_1}^{j}\right)\\
\vdots\\
\scriptstyle a_{\gamma+1-p}^{n_i}\mathcal{A}_{t}^{n_i}(\sum_{j=1}^{h}\lambda_{j}^{(n_i)}\chi_{r+1-\rho_t}^{j}+\sum_{j=1}^{\rho_t-1}\binom{n_i}{j}\sum_{g=1}^{h}\lambda_{g}^{(n_i)}\chi^{g}_{r+1-\rho_t+j})\\
\vdots\\
\scriptstyle a_{\gamma+1-p}^{n_i}\mathcal{A}_{t}^{n_i}\left(\sum_{j=1}^{h}\lambda_{j}^{(n_i)} \chi_{r}^{j}\right)\\
                                     \end{cases}\begin{array}{c}
\scriptstyle(L_1)\\
                                                 \vdots\\
\scriptstyle(L_{\rho_1})\\
\vdots\\
\scriptstyle(L_{r+1-\rho_t})\\
\vdots\\
\scriptstyle(L_r)\\
                                                \end{array}$$
Since $\mathcal{A}_j$ is an isometry, (\ref{equpol}) shows that for any $1\leq j\leq r$, $(L_j)\underset{i\to+\infty}{\longrightarrow}0$.
Hence, the first sum into $(\ref{equ})$ converges to 0 leading to:
$$(a_{\gamma+1-p}\mathcal{C}_{\gamma+1-p})^{n_i}(\sum_{j=h+1}^{k}\lambda_{j}^{(n_i)}y_j)\underset{i\to+\infty}{\longrightarrow}u.$$
Thus, $\Span\{y^{h+1}\ldots,y^k\}$ is supercyclic for $a_{\gamma+1-p} \mathcal{C}_{\gamma+1-p}$ and $\dim(\Span\{y^{h+1}\ldots,y^k\})\geq m_{\gamma+1-p}$.

Hence $h:=\dim(\Span\{z^1,\ldots,z^k\})\geq \sum_{i=\gamma+2-p}^{\gamma}m_i$ and $\dim(\Span\{y^{h+1},\ldots,y^{k}\})\geq m_{\gamma+1-p}$.
\medskip

Then, as the basis is reduced by Theorem \ref{theoreductiondebase}, we have $x^{j}=y^j\oplus0$ for any $h+1\leq j\leq k$, and so $k=\dim(\Span\{x^1,\ldots,x^k\})\geq \sum_{i=\gamma+2-p}^{\gamma}s_i+s_{\gamma+1-p}=k+1$. This contradiction proves that $T_p$ is not $\left((\sum_{i=\gamma+1-p}^{\gamma}m_i)-1\right)$-supercyclic.
\medskip

Let us focus now on the second part of the induction hypothesis.
For this purpose, let $M=\Span\{x^1,\ldots,x^b\}$ be a $b$-supercyclic subspace whose basis is reduced and $\sum_{i=\gamma+1-p}^{\gamma}m_i\leq b\leq t(p)$. Let also $T_{p}^{n_i}(\sum_{j=1}^{b}\lambda_{j}^{(n_i)}x^j)\underset{i\to+\infty}{\longrightarrow}0$.

For every $1\leq i\leq b$, decompose $x^{i}=y^{i}\oplus z^{i}$ relatively to the direct sum decomposition of $T_p$. Then, we just have to invoke Lemma \ref{lemconrec} ($a_{\gamma+1-p}= a$, $a_{\gamma+2-p}= b$, $b= m$, $t(p)-t(p-1)= N$, $h=$number of non-zero vectors among $z^1,\ldots,z^b$, $x^i=y^i$ , $\gamma=$number of blocks in $\mathcal{C}_{\gamma+1-p}$, $\mathcal{C}_{\gamma+1-p}= \mathcal{C}$, $a_{\gamma+1-p}\mathcal{C}_{\gamma+1-p}= T$) and the result comes: there exists $q\in\Z_{+}$ such that: $\frac{a_{\gamma+1-p}^{n_i}\lambda_{j}^{(n_i)}}{n_{i}^{q}}\underset{i\to+\infty}{\longrightarrow}0$ for any $1\leq j\leq b$.
\\This achieves the proof of the induction and of the theorem.
\end{proof}

We are now ready to state global results of supercyclicity for operators on $\R^N$. These results follow straightforwardly from Theorem \ref{theovpdiff} above and Theorem \ref{theodsbloc} and generalise Hilden and Wallen's and Herzog's results for supercyclic operators.

\begin{cor}\label{coravprinc}
 Let $N\geq2$ and $T$ be an operator on $\R^N$, then $T$ is not $(\rho-1)$ supercyclic.
\end{cor}

\begin{proof}
Without loss of generality, one may assume that $T$ is in Jordan form and also upon reordering blocks and considering a multiple of $T$ instead of $T$ itself, one may assume that the sequence formed with each Jordan block's modulus satisfies $\vert a_1\vert\leq\ldots\leq \vert a_\gamma\vert \leq1$.
As a consequence, one may realise $T$ as a direct sum of matrices $S_1,\ldots,S_t$, where $S_1$ contains all Jordan blocks with the smallest modulus and so on. Let $\rho^{(j)}$ denote the relative size of the matrix $S_j$, $j=1,\ldots,t$.
First, Theorem \ref{theodsbloc} implies that every matrix $S_j$ is no less than $\rho^{(j)}$-supercyclic. Then, one shall use Theorem \ref{theovpdiff} to come back to $T$, hence $T$ is no less than $\left(\sum_{j=1}^{t}\rho^{(j)}\right)$-supercyclic.
Then, one just have to recall the definition of $\rho$ and of $\rho^{(j)}$ providing $\left(\sum_{j=1}^{t}\rho^{(j)}\right)=\rho$. This proves the corollary.
\end{proof}

A direct application of the preceding corollary yields to a more concrete result:

\begin{cor}\label{corprincipal}
 Let $N\geq2$. There is no $(\lfloor\frac{N+1}{2}\rfloor-1)$-supercyclic operator on $\R^N$.
\\Moreover, there always exists a $(\lfloor\frac{N+1}{2}\rfloor)$-supercyclic operator on $\R^N$.
\end{cor}

\begin{proof}
The first part follows from Corollary \ref{coravprinc}. Indeed, if $N$ is even the lowest relative size of a matrix is $\frac{N}{2}$ and is $\frac{N+1}{2}$ if $N$ is odd. Thus the relative size of a matrix cannot be lower than $\lfloor\frac{N+1}{2}\rfloor$ on $\R^N$. The second part is due to Example \ref{exensup}.
\end{proof}

\begin{quest}
Does there exist a theorem similar to Theorem \ref{theovpdiff} in the case of a direct sum of Jordan blocks of modulus one?
\end{quest}

\begin{quest}
Does there exist a $(2N-2)$-supercyclic real Jordan block on $\R^{2N}$?
\\If so, what is the best supercyclic constant for a real Jordan block on $\R^{2N}$?
\end{quest}

\section{Strong $n$-supercyclicity}

The aim of this section is to study the existence of strong $n$-supercyclic operators in $\R^N$. Of course, this is interesting only if $k\leq N$. Bourdon, Feldman and Shapiro \cite{Bou} answer this question for the complex case. Indeed, they prove that $n$-supercyclicity cannot occur non-trivially in finite complex dimension, and thus strong $n$-supercyclicity cannot either.
However, in the real setting, we noticed in the previous section that $n$-supercyclicity can occur and thus the question for strong $n$-supercyclicity is still open.
For this purpose we need the following proposition from \cite{Ernststrongnsupop}. It provides a more concrete definition of strongly $n$-supercyclic operators:

\begin{prop}{(Proposition 1.13 \cite{Ernststrongnsupop})}\label{propbeq}
\textit{
Let $X$ be a completely separable Baire vector space. 
The following are equivalent:
\begin{enumerate}[(i)]
 \item $T$ is strongly $n$-supercyclic.
\item There exists an $n$-dimensional subspace $L$ such that for every $i\in\Z_{+}$, $T^{i}(L)$ is $n$-dimensional and $\mathcal{B}:=\cup_{i=1}^{\infty}\pi_{n}^{-1}(\tilde{T}^{i}(L))$ is dense in $X^{n}$.
\item There exists an $n$-dimensional subspace $L$ such that for every $i\in\Z_{+}$, $T^{i}(L)$ is $n$-dimensional and $\mathcal{E}:=\cup_{i=1}^{\infty}(T^{i}(L)\times\cdots\times T^{i}(L))$ is dense in $X^{n}$.
\end{enumerate}}
\end{prop}
\begin{rem}
Moreover, from the definition of strong $n$-supercyclicity or with the previous proposition, one may observe that if $T$ is a strongly $n$-supercyclic operator on $\R^{N}$ with $n\leq N$, then $T$ is bijective.
\end{rem}
We turn out to the case of strongly $n$-supercyclic operators on a real finite dimensional vector space.
Our first result is interesting and provides a partial answer to the question:

\begin{prop}\label{propfortedualite}
Let $n<N$.
An operator $T$ on $\mathbb{R}^{N}$ is strongly $n$-supercyclic if and only if $(T^{-1})^{\ast}$ is also strongly $(N-n)$-supercyclic and the strong $n$-supercyclic subspaces for $T$ are orthogonal to the strongly $(N-n)$-supercyclic subspaces of $(T^{-1})^{\ast}$.
\end{prop}

This duality property is very useful if one combines it with Corollary \ref{corprincipal}, by the way we get:

\begin{cor}\label{corfortedualite}
There is no strongly $n$-supercyclic operators on $\mathbb{R}^{2N+1}$ for any $N\geq 1$ and any $1\leq n<N$.
\\There is no strongly $n$-supercyclic operators on $\mathbb{R}^{2N}$ for any $N\geq 2$ and any $1\leq n<N$, $n\neq\frac{N}{2}$.
\end{cor}

This corollary provides many examples of $n$-supercyclic operators that are not strongly $n$-supercyclic and thus answers the question of Shkarin \cite{Shkauniv} by proving that $n$-supercyclicity and strong $n$-supercyclicity are not equivalent. 
\begin{exe}{A 2-supercyclic operator that is not strongly 2-supercyclic}
\\
Any rotation on $\R^3$ around any one-dimensional subspace and with angle linearly independent with $\pi$ on $\Q$ is 2-supercyclic but not strongly 2-supercyclic.
\end{exe}

We are now turning to the proof of Proposition \ref{propfortedualite}. We need the two following well-known lemmas:

\begin{lem}\label{lemfortedualitelem1}
Let $M$  be a subspace of $\mathbb{R}^{N}$ and let $T$ be an automorphism on $\mathbb{R}^{N}$.
Then, for any $i\in\mathbb{N}$, $(T^{i}(M))^{\perp}=(T^{-i})^{\ast}(M^{\perp})$.
\end{lem}

\begin{lem}\label{lemfortedualitelem2}
Let $\Phi: \mathbb{P}_{n}(\mathbb{R}^{N})\rightarrow \mathbb{P}_{N-n}(\mathbb{R}^{N})$ be defined by the formula $\Phi(M)=M^{\perp}$ for every $M\in\mathbb{P}_{n}(\R^{N})$. Then $\Phi$  is a homeomorphism.
\end{lem}

The first lemma is a classical one and the second may be found and proven in \cite{Mil}.
Then, the proof of the proposition is straightforward:\\
\renewcommand{\proofname}{Proof of Proposition \ref{propfortedualite}}
\begin{proof}
The combination of Lemma \ref{lemfortedualitelem1} and Lemma \ref{lemfortedualitelem2} when $M$ is a strongly $n$-supercyclic subspace for $T$, implies that $\Phi(\{T^{i}(M)\}_{i\in\mathbb{N}})=\{(T^{-i})^{\ast}(M^{\perp})\}_{i\in\mathbb{N}}$ is dense in $\mathbb{P}_{N-n}(\mathbb{R}^{N})$.
\end{proof}
\renewcommand{\proofname}{Proof}
\subsection{Strongly 2-supercyclic operators on $\R^4$}
The idea is to prove by induction that there is no strongly $n$-supercyclic operator on $\R^N$ for $N\geq3$ and $1\leq n<N$. The first step is to prove this for $N$ small. It is already done for $N=3$ by Corollary \ref{corfortedualite}, we now focus on the case $N=4$. To this purpose, we begin with characterising 2-supercyclic subspaces for a direct sum of rotations and we prove then that none are strongly 2-supercyclic.

\begin{prop}\label{propesp2sup}

Let $R$ be a direct sum of two rotations: $$R=\left(\begin{array}{cc}
R_{\theta_{1}}&\multicolumn{1}{|c}{0} \\\cline{1-2} 
0&\multicolumn{1}{|c}{R_{\theta_{2}}}\\
\end{array}\right)\text{ with $\{\theta_{1},\theta_{2},\pi\}$ linearly independent over $\Q$.}$$
Then:
$$\ES_{2}(R)=\left\{ \Span\left\{\left(\begin{array}{c}
      x_{1} \\
      x_{2} \\
      ay_{1} \\
      ay_{2}\\
   \end{array}\right),\left(\begin{array}{c}
      bx_{1} \\
      bx_{2} \\
      y_{1} \\
      y_{2}\\
   \end{array}\right)\right),\left(\begin{array}{c}
      x_{1} \\
      x_{2} \\
   \end{array}\right),\left(\begin{array}{c}
      y_{1} \\
      y_{2}\\
   \end{array}\right\}\in\mathbb{R}^{2}\setminus\{0\},ab\neq1\right\}.$$
\end{prop}

\begin{proof}
 We proceed by double inclusion.

First, let $x=\left(\begin{array}{c}
      x_{1} \\
      x_{2} \\
      ay_{1} \\
      ay_{2}\\
   \end{array}\right)$ and $y=\left(\begin{array}{c}
      bx_{1} \\
      bx_{2} \\
      y_{1} \\
      y_{2}\\
   \end{array}\right)$ satisfy the above hypothesis, we are going to prove that the span of these two vectors is $2$-supercyclic for R.\\
Remark that $x-ay$ is non-zero but its third and fourth components are null and $y-bx$ is also non-zero but its two first components are null.\\ 
Let  $U$ and $V$ be two non-empty open sets in $\R^{2}$.
As $\theta_{1},\theta_{2},\pi$ are linearly independent over $\Q$, there exist $i\in\N$ and $c_{1},c_{2}\in\R$ such that
$$c_{1}(1-ab)R_{\theta_{1}}^{i}\left(\begin{array}{c}x_1\\x_2\\\end{array}\right)\in U\text{ and }c_{2}(1-ab)R_{\theta_{2}}^{i}\left(\begin{array}{c}y_1\\y_2\\\end{array}\right)\in V.$$
Thus $R^{i}(c_{1}(x-ay)+c_{2}(-bx+y))\in U\times V$, hence  $\Span\{x,y\}$ is a 2-supercyclic subspace of $R$.

Now for the converse, let us suppose that there exists a two-dimensional subspace $M=\Span\left\{x,y\right\}$ which is 2-supercyclic for $R$ and which is not satisfying the conditions stated in the proposition. Hence, either the family $\left\{\left(\begin{array}{c}
      x_{1} \\
      x_{2} \\
   \end{array}\right),\left(\begin{array}{c}
      y_{1} \\
      y_{2}\\
   \end{array}\right)\right\}$ is linearly independent or we can assume that $\left(\begin{array}{c}y_{1}\\y_{2}\end{array}\right)=\left(\begin{array}{c}0\\0\end{array}\right)$.
Set $\mathcal{C}_{1}:=\{z\in\R^{2}:\Vert z\Vert<1\}$ and define also $\mathcal{C}_{2}:=\{z\in\R^{2}:\Vert z\Vert>t\}$ where $t$ is a positive real number we may choose later.
Then, as $M$ is a 2-supercyclic subspace for $R$, $\{R^{i}(\lambda x+\mu y)\}_{i\in\N,(\lambda,\mu)\in\R^{2}}$ is dense in $\R^{4}$.
As a result, there exist $i\in\N$ and $(\lambda,\mu)\in\R^{2}$ such that 
$$
   R^{i}(\lambda x+\mu y)\in \mathcal{C}_{1}\times \mathcal{C}_{2}
\Leftrightarrow\left(\begin{array}{c}
                      \lambda\left(\begin{array}{c}
                                    x_{1}\\
x_{2}\\
                                   \end{array}\right)+\mu\left(\begin{array}{c}
                                    y_{1}\\
y_{2}\\
                                   \end{array}\right)\\
\lambda\left(\begin{array}{c}
                                    x_{3}\\
x_{4}\\
                                   \end{array}\right)+\mu\left(\begin{array}{c}
                                    y_{3}\\
y_{4}\\
                                   \end{array}\right)\\

                     \end{array}\right)\in R^{-i}_{\theta_{1}}(\mathcal{C}_{1})\times R^{-i}_{\theta_{2}}(\mathcal{C}_{2})=\mathcal{C}_{1}\times\mathcal{C}_{2}
$$

Then, define $$\begin{aligned}
          \Gamma_{1}&:=\left\{(\lambda,\mu)\in\R^{2}:  \lambda\left(\begin{array}{c}
                                    x_{1}\\
				    x_{2}\\
                                   \end{array}\right)+\mu\left(\begin{array}{c}
                                    y_{1}\\
				    y_{2}\\
                                   \end{array}\right)\in\mathcal{C}_{1}
 \right\}&\\
\end{aligned}$$

Since $\mathcal{C}_{1}$ is bounded and the family $\left\{\left(\begin{array}{c}
                                    x_{1}\\
				    x_{2}\\
                                   \end{array}\right);\left(\begin{array}{c}
                                    y_{1}\\
				    y_{2}\\
                                   \end{array}\right)\right\}$ is linearly independent or $\left(\begin{array}{c}y_{1}\\y_{2}\end{array}\right)=\left(\begin{array}{c}0\\0\end{array}\right)$ then one may deduce that $\Gamma_{1}$ is bounded too.
Note that $\Omega:=\left\{\lambda\left(\begin{array}{c}
                                    x_{3}\\
				    x_{4}\\
                                   \end{array}\right)+\mu\left(\begin{array}{c}
                                    y_{3}\\
				    y_{4}\\
                                   \end{array}\right),(\lambda,\mu)\in\Gamma_{1}\right\}$ is also bounded because of the boundedness of $\Gamma_{1}$. We now define $t$ being an upper bound for $\Omega$, then one deduces that $\mathcal{C}_{2}\cap\Omega=\emptyset$. This contradicts the fact that $M$ is a 2-supercyclic subspace for $R$.
So $M$ satisfies the proposition's conditions.\\
\end{proof}

\begin{cor}\label{cordoublerot}

$R$ is not strongly 2-supercyclic on $\mathbb{R}^{4}$.
\end{cor}

\begin{proof}
Assume that $R$ is strongly 2-supercyclic on $\mathbb{R}^{4}$, then any strongly 2-supercyclic subspace is given by Proposition \ref{propesp2sup}.
Thus if $x,y\in\R^4$ span a strongly 2-supercyclic subspace for $R$, then
$$R^{i}(\lambda x +\mu y)=\left(\begin{array}{c}
      (\lambda+b\mu)R^{i}_{\theta_{1}}\left(\begin{array}{c}
      x_{1} \\
      x_{2}\\
   \end{array}\right) \\
     (\lambda a+\mu)R^{i}_{\theta_{2}}\left(\begin{array}{c}
      y_{1} \\
      y_{2}\\
   \end{array}\right) \\
   \end{array}\right).$$
Moreover, according to Proposition \ref{propbeq} this means that for any two non-empty open sets $U_{1},U_{2}\subset\mathbb{R}^{2}$, there exist $i\in\mathbb{N}$ and $\lambda,\mu,\alpha,\beta\in\mathbb{R}$ such that:
$$\begin{array}{ccc}
(\lambda+b\mu)R^{i}_{\theta_{1}}\left(\begin{array}{c}
      x_{1} \\
      x_{2}\\
   \end{array}\right)\in U_{1}
&\text{ and }&(\alpha+b\beta)R^{i}_{\theta_{1}}\left(\begin{array}{c}
      x_{1} \\
      x_{2}\\
   \end{array}\right)\in U_{2}.
\end{array}$$
But this cannot happen if we choose $U_{1}$ and $U_{2}$ such that there does not exist a straight line passing through the origin and intersecting both $U_1$ and $U_2$. Therefore, $R$ is not strongly 2-supercyclic. 
\end{proof}

We are now going to deal with two other different-shaped operators on $\R^4$.

\begin{prop}\label{propinitrec4}
$\left(\begin{array}{cc}
\multicolumn{1}{c|}{A}  &A \\ \cline{1-2}
0 &\multicolumn{1}{|c}{A} \\
\end{array}\right)$ and $\left(\begin{array}{ccc}
\multicolumn{1}{c|}{A}   &0 \\ \cline{1-2}
0 &\multicolumn{1}{|c}{B} \\
\end{array}\right)$ with $A=\left(\begin{array}{cc}
a  &-b \\ 
b &a \\
\end{array}\right)$ and $B=\left(\begin{array}{cc}
c  &-d \\ 
d &c \\
\end{array}\right)$, $(a,b,c,d)\in\mathbb{R}^{4}$ are not strongly $2$-supercyclic.
\end{prop}

\begin{proof}
Without loss of generality, one may assume $ \left(\begin{array}{c}a\\b\\\end{array}\right)\neq\left(\begin{array}{c}0\\0\\\end{array}\right)$ and $ \left(\begin{array}{c}c\\d\\\end{array}\right)\neq\left(\begin{array}{c}0\\0\\\end{array}\right)$ because strongly 2-supercyclic operators have dense range.
\\If $R=\left(\begin{array}{cc}
\multicolumn{1}{c|}{A}  &A \\ \cline{1-2}
0 &\multicolumn{1}{|c}{A} \\
\end{array}\right)$, then $R$ is not strongly 2-supercyclic according to Proposition \ref{propjordpas2} .
\\If $R=\left(\begin{array}{ccc}
\multicolumn{1}{c|}{A}   &0 \\ \cline{1-2}
0 &\multicolumn{1}{|c}{B} \\
\end{array}\right)$, we assume, in order to obtain a contradiction, that $R$ is strongly $2$-supercyclic.
Upon considering a scalar multiple, relabelling and rearranging blocks, one can suppose: $R=\left(\begin{array}{cc}
\multicolumn{1}{c|}{R_{\theta}}  &0 \\ \cline{1-2}
0 &\multicolumn{1}{|c}{C} \\
\end{array}\right)$ with $C=\left(\begin{array}{cc}
c  &-d \\
d &c \\
\end{array}\right)$ and $c^{2}+d^{2}\leq1$.
Let $M=\Span\{x,y\}$ be a strongly $2$-supercyclic subspace for $R$. 
\\If $c^{2}+d^{2}=1$, then Corollary \ref{cordoublerot} implies that $R$ is not strongly $2$-supercyclic.
\\If $c^{2}+d^{2}<1$,
since $M$ is strongly $2$-supercyclic for $R$ and using Proposition \ref{propbeq} then for any non-empty open sets $U_{1},U_{2},V_{1},V_{2}$ in $\mathbb{R}^{2}$, there exist $i\in\mathbb{N}$ and $(\lambda,\mu,\alpha,\beta)\in\mathbb{R}^{4}$ such that:
$$\begin{aligned}
\begin{cases}
R^{i}(\lambda x+\mu y)\in U_{1}\times U_{2}\\
R^{i}(\alpha x+\beta y)\in V_{1}\times V_{2}\\
\end{cases}
\Leftrightarrow
&\begin{cases}
\lambda \left(\begin{array}{c}x_{1}\\x_{2}\\\end{array}\right)+\mu \left(\begin{array}{c}y_{1}\\y_{2}\\\end{array}\right)\in R^{-i}_{\theta}(U_{1})\\
\alpha \left(\begin{array}{c}x_{1}\\x_{2}\\\end{array}\right)+\beta \left(\begin{array}{c}y_{1}\\y_{2}\\\end{array}\right)\in R^{-i}_{\theta}(V_{1})\\
\lambda \left(\begin{array}{c}x_{3}\\x_{4}\\\end{array}\right)+\mu \left(\begin{array}{c}y_{3}\\y_{4}\\\end{array}\right)\in C^{-i}(U_{2})\\
\alpha \left(\begin{array}{c}x_{3}\\x_{4}\\\end{array}\right)+\beta \left(\begin{array}{c}y_{3}\\y_{4}\\\end{array}\right)\in C^{-i}(V_{2})\\
\end{cases}
\end{aligned}$$
From this last identity, we deduce that the family $\left\{\left(\begin{array}{c}x_{1}\\x_{2}\\\end{array}\right),\left(\begin{array}{c}y_{1}\\y_{2}\\\end{array}\right)\right\}$ is linearly independent. Indeed, if not, then we choose $U_{1}$ and $V_{1}$ such that there does not exist a straight line passing through the origin and intersecting both $U_1,V_1$ to obtain a contradiction with the previous identity.
\\The family being linearly independent, let $U_{1}=\{z\in\R^2: \Vert z\Vert<1\}$ and $U_{2}=\{z\in\R^2: \Vert z\Vert>t\}$.
Set also $$\Gamma=\left\{(\lambda,\mu)\in\mathbb{R}^{2}:\lambda \left(\begin{array}{c}x_{1}\\x_{2}\\\end{array}\right)+\mu \left(\begin{array}{c}y_{1}\\y_{2}\\\end{array}\right)\in U_1\right\}.$$
Since the family $\left\{\left(\begin{array}{c}x_{1}\\x_{2}\\\end{array}\right),\left(\begin{array}{c}y_{1}\\y_{2}\\\end{array}\right)\right\}$ is linearly independent and $U_1$ is a bounded set, one may deduce that $\Gamma$ is bounded in $\mathbb{R}^{2}$ and then $\Omega=\left\{\lambda \left(\begin{array}{c}x_{3}\\x_{4}\\\end{array}\right)+\mu \left(\begin{array}{c}y_{3}\\y_{4}\\\end{array}\right), (\lambda,\mu)\in\Gamma\right\}$ is obviously bounded too and we define $t$ as an upper bound for $\Omega$.\\
On the other hand, we have that $C^{-1}=\frac{1}{c^{2}+d^{2}}\left(\begin{array}{cc}
c  &d \\ 
-d &c \\
\end{array}\right)=\frac{1}{\sqrt{c^{2}+d^{2}}}R_{\varphi}$, thus $C^{-i}(U_2)\subseteq U_2$ for any $i\in\N$. Hence, $\Omega\cap\cup_{i\in\N}C^{-i}(U_2)=\emptyset$, contradicting the strong 2-supercyclicity of $R$.
\end{proof}

\subsection{General result}
We are going to prove that there is no non-trivial strongly $n$-supercyclic operator by induction on the space dimension. The following proposition inspired from Bourdon, Feldman, Shapiro \cite{Bou} is useful for the induction step:

\begin{prop}\label{propHR}
Let $X$ be a Hausdorff topological vector space. Let also $T:X\to X$ be a bounded operator and $K$ a closed invariant subspace for $T$. If $T$ is strongly $n$-supercyclic then the quotient map $T_{K}:\frac{X}{K}\to \frac{X}{K}$ is strongly $n$-supercyclic too.
\end{prop}

The proof is just a verification using the characterisation of strong $n$-supercyclicity given in Proposition \ref{propbeq}.

\begin{theo}\label{theononexistencedimfinie}

For any $N\geq3$, there is no strongly $n$-supercyclic operator on $\mathbb{R}^{N}$ for $1\leq n<N$.
\end{theo}

\begin{proof}
We give a proof by induction on the space dimension.
First, according to Corollary \ref{corfortedualite} and Herzog's result \cite{Hilwal}, there is no strongly $n$-supercyclic operator on $\mathbb{R}^{3}$ with $n=1,2$.
\\So, let $N\geq 4$.
We want to prove that there is no strongly $n$-supercyclic operator on $\R^N$ for $n<N$.
By Corollary \ref{corfortedualite}, we can suppose that $N\geq4$ is even and $n=\frac{N}{2}$. Then, let also $R$ be an operator on $\mathbb{R}^{N}$, using the Jordan real decomposition one may suppose:
$$R=\left(\begin{array}{ccccccccc}
\cline{1-1}
\multicolumn{1}{|c|}{J_{1}}&0&\cdots&\cdots&\cdots&0 \\ \cline{1-1}
0 &\ddots&0&\cdots&\cdots&0 \\ \cline{3-3}
0 &0&\multicolumn{1}{|c|}{J_{q}}&0&\cdots&0 \\ \cline{3-4}
0 &\cdots&0&\multicolumn{1}{|c|}{\mathcal{J}_{1}}&0&0 \\ \cline{4-4}
0 &\cdots&\cdots&0&\ddots&0 \\ \cline{6-6}
0 &0&0&0&0&\multicolumn{1}{|c|}{\mathcal{J}_{r}} \\ \cline{6-6}
\end{array}\right)$$
where $J_{i}$ is a classical Jordan block:
$$J_{i}=\left(\begin{array}{ccccc}
\mu_{i}   &\mu_{i}&0  &0 \\
0&\ddots &\ddots&\cdots\\
\vdots&0&\ddots&\mu_{i} \\
0 & \cdots      &0&\mu_{i} \\
\end{array}\right)$$
and $\mathcal{J}_{i}$ is a real Jordan block:
$$\mathcal{J}_{i}=\left(\begin{array}{cccc}
\multicolumn{1}{c|}{\mathcal{A}_{i}}   &\multicolumn{1}{|c|}{\mathcal{A}_{i}}&0  &0 \\ \cline{1-2}
0&\ddots &\ddots&\cdots\\ \cline{4-4}
\vdots&0&\ddots&\multicolumn{1}{|c}{\mathcal{A}_{i}} \\ \cline{4-4}
0 & \cdots      &0&\multicolumn{1}{|c}{\mathcal{A}_{i}} \\
\end{array}\right).$$
By contradiction, assume that $M$ is a strongly $\frac{N}{2}$-supercyclic subspace for $R$.
Then we have to consider two different cases, either $q=0$ or $q\neq 0$.

If $q\neq 0$, then $K=\Span\{(1,0,\ldots,0)\}$ is $R$-invariant. Consider then the quotient $\frac{\R^N}{K}$ and apply Proposition \ref{propHR} to deduce that $R_{K}$ is strongly $\frac{N}{2}$-supercyclic on $\mathbb{R}^{N-1}$. In addition, as $N\geq4$ then the following inequalities hold $1\leq \frac{N}{2}<N-1$ but this contradicts Corollary \ref{corfortedualite} because $N-1$ is odd.

Then, if $q=0$, there are two cases, either $N=4$ or $N\geq6$.
\\If $N=4$, then Corollary \ref{cordoublerot} and Proposition \ref{propinitrec4} implies that $R$ is not strongly $2$-supercyclic.
\\If $N\geq6$, notice that $K=\Span\{(1,0,\ldots,0),(0,1,0,\ldots,0)\}$ is $R$-invariant. One may consider the quotient by $K$, and apply Proposition \ref{propHR} providing that $R_{K}$ is strongly $\frac{N}{2}$-supercyclic on $\mathbb{R}^{N-2}$. Moreover, $1\leq\frac{N}{2}\neq \frac{N-2}{2}=\frac{N}{2}-1<N-2$. This contradicts the induction hypothesis.
These contradictions prove that there is no strongly $n$-supercyclic operator on $\R^N$ with $1\leq n<N$.
\end{proof}

\bibliographystyle{plain}

 \scshape

\vglue0.3cm
\hglue0.02\linewidth\begin{minipage}{0.9\linewidth}
Romuald Ernst\\
{Laboratoire de~Math\'ematiques (UMR 6620)}\\
Universit\'e Blaise Pascal\\
Complexe universitaire des C\'ezeaux\\
63177 Aubi\`ere Cedex, France \\
E-mail : \parbox[t]{0.45\linewidth}{\texttt{Romuald.Ernst@math.univ-bpclermont.fr}}
\end{minipage}

\end{document}